\documentclass[12pt]{article}
\usepackage{a4wide}
\usepackage{amsthm}
\usepackage{amsfonts}
\usepackage{amssymb}
\usepackage{amsmath}
\usepackage{stmaryrd}
\usepackage{cite}
\usepackage{epsfig}
\usepackage{array}
\usepackage{booktabs}

\newtheorem{theorem}{Theorem}
\newtheorem{lemma}[theorem]{Lemma}
\newtheorem{proposition}[theorem]{Proposition}
\newtheorem{conjecture}{Conjecture}
\def\svejk{\v{S}vejk }
\def\CC{{\mathcal C}}
\def\JJ{{\mathcal J}}
\def\CF{{\rm CF}}
\def\dd{\;\mbox{d}}
\def\NN{{\mathbb N}}
\def\RR{{\mathbb R}}

\newcommand{\scalar}[2]{\left\langle #1,#2\right\rangle}

\newcommand{\trunc}[1]{\operatorname{trunc}\left(#1\right)}

\begin{document}

\title{Weak regularity and finitely forcible graph limits~\thanks{The work of the first and third authors leading to this invention has received funding from the European Research Council under the European Union's Seventh Framework Programme (FP7/2007-2013)/ERC grant agreement no.~259385.}}
\author{Jacob W.~Cooper\thanks{Department of Computer Science, University of Warwick, Coventry CV4 7AL, UK. E-mail: {\tt j.w.cooper@warwick.ac.uk}.}\and
        Tom\'a\v{s} Kaiser\thanks{Department of Mathematics, Institute for Theoretical Computer Science (CE-ITI) and the European Centre of Excellence NTIS (New Technologies for the Information Society), University of West Bohemia, Univerzitn\'{\i}~8, 306~14~Plze\v{n}, Czech Republic. E-mail: {\tt kaisert@kma.zcu.cz}. This author was supported by the grant GA14-19503S (Graph coloring and structure) of the Czech Science Foundation.}\and
        Daniel Kr\'al'\thanks{Mathematics Institute, DIMAP and Department of Computer Science, University of Warwick, Coventry CV4 7AL, UK. E-mail: {\tt d.kral@warwick.ac.uk}. The work of this author was also supported by the Engineering and Physical Sciences Research Council Standard Grant number EP/M025365/1.}\and
	Jonathan A.~Noel\thanks{Mathematical Institute, University of Oxford, Oxford OX2 6GG, UK. E-mail: {\tt noel@maths.ox.ac.uk}.}}
\date{}	
\maketitle

\begin{abstract}	
Graphons are analytic objects representing limits of convergent sequences of graphs.
Lov\'asz and Szegedy conjectured that every finitely forcible graphon,
i.e.~any graphon determined by finitely many graph densities, has a simple structure.
In particular, one of their conjectures would imply that
every finitely forcible graphon has a weak $\varepsilon$-regular partition with the number of parts bounded by a polynomial in $\varepsilon^{-1}$.
We construct a finitely forcible graphon $W$ such that
the number of parts in any weak $\varepsilon$-regular partition of $W$ is at least
exponential in $\varepsilon^{-2}/2^{5\log^*\varepsilon^{-2}}$. 
This bound almost matches the known upper bound for graphs and, in a certain sense, is the best possible for graphons.
\end{abstract}	

\section{Introduction}
\label{sect-intro}

The theory of combinatorial limits has recently attracted a significant amount of attention.
This line of research
was sparked by limits of dense graphs~\cite{bib-borgs08+,bib-borgs+,bib-borgs06+,bib-lovasz06+},
which we focus on here, followed by limits of other structures,
e.g.~permutations~\cite{bib-hoppen-lim1,bib-hoppen-lim2,bib-kral12+},
sparse graphs~\cite{bib-bollobas11+,bib-elek07} and
partial orders~\cite{bib-janson11}.
Methods related to combinatorial limits have led to substantial results in many areas of mathematics and computer science,
particularly in extremal combinatorics.
For example, the notion of flag algebras, which is strongly related to combinatorial limits,
resulted in progress on many important problems in extremal combinatorics~\cite{bib-flag1, bib-flag2, bib-flagrecent, bib-flag3, bib-flag4, bib-flag5, bib-flag6, bib-flag7, bib-flag8, bib-flag9, bib-flag10, bib-razborov07,bib-flag11, bib-flag12}.
Theory of combinatorial limits also provided a new perspective on existing concepts in other areas,
e.g. property testing algorithms in computer science~\cite{bib-glebov-test,bib-hoppen-test,bib-lovasz10+}.

A convergent sequence of dense graphs can be represented by an analytic object called a \emph{graphon}.
Let $d(H,W)$ be the density of a graph $H$ in a graphon $W$ (a formal definition is given in Section~\ref{sec:limits}).
A graphon $W$ is said to be \emph{finitely forcible} if it is determined by finitely many graph densities,
i.e.~there exist graphs $H_1,\ldots,H_k$ and reals $d_1,\ldots,d_k$ such that $W$ is the unique graphon with $d(H_i,W)=d_i$.
Finitely forcible graphons appear in many different settings, one of which is in extremal combinatorics.
It is known that if a graphon is finitely forcible, then it is the unique graphon which minimizes a fixed finite linear combination of subgraph densities,
i.e.~finitely forcible graphons are extremal points of the space of all graphons.
The following conjecture~\cite[Conjecture 7]{bib-lovasz11+} claims that the converse is also true.
\begin{conjecture}
Let $H_1,\ldots,H_k$ be finite graphs and $\alpha_1,\ldots,\alpha_k$ reals.
There exists a finitely forcible graphon $W$ that minimizes the sum $\sum\limits_{i=1}^k \alpha_i d(H_i,W)$.
\end{conjecture}
\noindent Finitely forcible graphons are also related to quasirandomness in graphs
as studied e.g.~by Chung, Graham and Wilson~\cite{bib-chung89+}, R\"odl~\cite{bib-rodl}, and Thomason~\cite{bib-thomason, bib-thomason2}.
In the language of graph limits, results on quasirandom graphs state that every constant graphon is finitely forcible.
A generalization of this statement was proven by Lov\'asz and S\'os~\cite{bib-lovasz08+}:
every step graphon (i.e.~a multipartite graphon with a finite number of parts and uniform edge densities between its parts) is finitely forcible.

In~\cite{bib-lovasz11+},
Lov\'asz and Szegedy carried out a more systematic study of finitely forcible graphons.
The examples of finitely forcible graphons that they constructed led to a belief that
finitely forcible graphons must have a simple structure.
To formalize this, they introduced the (topological) space $T(W)$ of typical vertices of a graphon $W$ and
conjectured the following~\cite[Conjectures~9 and 10]{bib-lovasz11+}.
\begin{conjecture}
\label{conj-compact}
The space of typical vertices of every finitely forcible graphon is compact.
\end{conjecture}
\begin{conjecture}
\label{conj-dimension}
The space of typical vertices of every finitely forcible graphon has a finite dimension.
\end{conjecture}
\noindent Both conjectures were disproved through counterexample constructions in~\cite{bib-inf,bib-comp}.

Conjecture~\ref{conj-dimension} is a starting point of our work.
Analogously to weak regularity of graphs,
every graphon has a weak $\varepsilon$-regular partition with at most $2^{O\left(\varepsilon^{-2}\right)}$ parts.
(See Section~\ref{sec:weak} for the necessary definitions.)
If the space of typical vertices of a graphon is equipped with an appropriate metric,
then its Minkowski dimension is linked to the number of parts in its weak regular partitions.
In particular, if its Minkowski dimension is $d$,
then the graphon has weak $\varepsilon$-regular partitions with $O\left(\varepsilon^{-d}\right)$ parts.
Consequently, if Conjecture~\ref{conj-dimension} were true,
the number of parts of a weak $\varepsilon$-regular partitions of a finitely forcible graphon would be bounded by a polynomial of $\varepsilon^{-1}$.
The number of parts in weak $\varepsilon$-regular partitions of a graphon constructed in~\cite{bib-inf} as a counterexample to Conjecture~\ref{conj-dimension}
is $2^{\Theta(\log^2\varepsilon^{-1})}$, which is superpolynomial in $\varepsilon^{-1}$,
but is much smaller than the general upper bound of $2^{O\left(\varepsilon^{-2}\right)}$.
We construct a finitely forcible graphon almost matching the upper bound.
\begin{theorem}
\label{thm-main}
There exist a finitely forcible graphon $W$ and positive reals $\varepsilon_i$ tending to $0$
such that every weak $\varepsilon_i$-regular partition of $W$
has at least $2^{\Omega\left(\varepsilon_i^{-2}/2^{5\log^*\varepsilon_i^{-2}}\right)}$ parts.
\end{theorem}
\noindent As pointed out to us by Jacob Fox,
there is no graphon (finitely forcible or not) matching the upper bound $2^{O\left(\varepsilon^{-2}\right)}$
for infinitely many values of $\varepsilon$ tending to $0$.
In light of this, Theorem~\ref{thm-main} is almost the best possible.
\begin{proposition}
\label{prop-lower-graphon}
There exist no graphon $W$, a positive real $c$ and positive reals $\varepsilon_i$ tending to $0$ such that
every weak $\varepsilon_i$-regular partition of $W$ has at least $2^{c\varepsilon_i^{-2}}$ parts.
\end{proposition}
\noindent The proof of this proposition is sketched at the end of Section~\ref{sec:weak}.

We will refer to the graphon $W$ from Theorem~\ref{thm-main} as the \svejk graphon. \svejk is the name of a famous (and fictitious) brave Czech soldier and, more importantly for us, it is the name of the restaurant where we usually ate lunch during our work on this subject
while three of us were visiting the University of West Bohemia in Pilsen.

\section{Graph limits}
\label{sec:limits}

We now introduce notions related to graphons and convergent sequences of graphs.
The {\em density} of a graph $H$ in $G$, which is denoted by $d(H,G)$,
is the probability that $|H|$ randomly chosen vertices of $G$ 
induce a subgraph isomorphic to $H$, where $|H|$ is the order (the number of vertices) of $H$.
A sequence of graphs $(G_n)_{n\in\NN}$ with the number of their vertices tending to infinity
is {\em convergent} if the sequence $d(H,G_n)$ converges for every graph $H$.
Note that if $G_n$ has $o(|G_n|^2)$ edges,
then the sequence $(G_n)_{n\in\NN}$ is convergent for trivial reasons.
Hence, this notion of graph convergence is of interest for sequences of dense graphs,
i.e.~graphs with $\Omega(|G_n|^2)$ edges.

A convergent sequence of dense graphs can be represented by an analytic object called a graphon.
A {\em graphon} $W$ is a measurable function from $[0,1]^2$ to $[0,1]$ that is symmetric,
i.e. it holds that $W(x,y)=W(y,x)$ for every $x,y\in [0,1]$.
The points in $[0,1]$ are often referred to as the {\em vertices} of the graphon $W$.

If $W$ is a graphon, then a {\em $W$-random graph} of order $k$ is obtained by sampling $k$ random points $x_1,\ldots,x_k\in [0,1]$
uniformly and independently, and joining the $i$-th and the $j$-th vertex by an edge with probability $W(x_i,x_j)$.
The {\em density} of a graph $H$ in a graphon $W$ is the probability that a $W$-random graph of order $|H|$ is isomorphic to $H$.
If $(G_n)_{n\in\NN}$ is a convergent sequence of graphs,
then there exists a graphon $W$ such that $d(H,W)=\lim\limits_{n\to\infty} d(H,G_n)$ for every graph $H$~\cite{bib-lovasz06+}.
This graphon can be viewed as the limit of the sequence $(G_n)_{n\in\NN}$.
On the other hand,
a sequence of $W$-random graphs of increasing orders is convergent with probability one and
its limit is the graphon $W$.

Two graphons $W_1$ and $W_2$ are {\em weakly isomorphic}
if there exist measure preserving maps $\varphi_1$ and $\varphi_2$ from $[0,1]$ to $[0,1]$ such that
$W_1(\varphi_1(x),\varphi_1(y))=W_2(\varphi_2(x),\varphi_2(y))$ for almost every pair $(x,y)\in [0,1]^2$.
If two graphons $W_1$ and $W_2$ are weakly isomorphic, then $d(H,W_1)=d(H,W_2)$ for every graph $H$.
The converse is also true~\cite{bib-borgs10+}:
if two graphons $W_1$ and $W_2$ satisfy that $d(H,W_1)=d(H,W_2)$ for every graph $H$,
then $W_1$ and $W_2$ are weakly isomorphic.
Hence, the limit of a convergent sequence of graphs is unique up to weak isomorphism.
We finish with giving a formal definition of a finitely forcible graphon:
a graphon $W$ is {\em finitely forcible} if there exist graphs $H_1,\ldots,H_k$ such that
if a graphon $W'$ satisfies that $d(H_i,W')=d(H_i,W)$ for $i=1,\ldots,k$, then $d(H,W')=d(H,W)$ for every graph $H$.

\section{Weak regular partitions}
\label{sec:weak}

In this section, we recall some basic concepts related to weak regularity for graphs and graphons and
cast the lower bound construction of Conlon and Fox from~\cite{bib-conlon12+} in the language of graphons.
Since we do not use any other type of regularity partition,
we will just say ``regular'' instead of ``weak regular'' in what follows.

We start with defining the notion for graphs.
If $G$ is a graph and $A$ and $B$ two subsets of its vertices,
let $e(A,B)$ be the number of edges $uv$ such that $u\in A$ and $v\in B$.
A partition of a vertex set $V(G)$ of a graph $G$ into subsets $V_1,\ldots,V_k$
is said to be {\em $\varepsilon$-regular}
if it holds that
$$\left|e(A,B)-\sum_{i,j\in [k]}\frac{e(V_i,V_j)}{|V_i||V_j|}|V_i\cap A||V_j\cap B|\right|\le\varepsilon |V(G)|^2$$
for every two subsets $A$ and $B$ of $V(G)$.
It is known that for every $\varepsilon>0$,
there exists $k_0\le 2^{O\left(\varepsilon^{-2}\right)}$ (which depends on $\varepsilon$ only) such that
every graph has an $\varepsilon$-regular partition with at most $k_0$ parts~\cite{bib-frieze99+}.
This dependence of $k_0$ on $\varepsilon$ is best possible up to a constant factor in the exponent as shown by Conlon and Fox~\cite{bib-conlon12+}.

We now define the analogous notion for graphons.
Let $W:[0,1]^2\to [0,1]$ be a graphon. If $A$ and $B$ are two measurable subsets of $[0,1]$,
then the {\em density $d_W(A,B)$ between $A$ and $B$} is defined to be
$$d_W(A,B)=\int\limits_{A\times B} W(x,y)\dd x\dd y\;\mbox{.}$$
We will omit $W$ in the subscript if the graphon $W$ is clear from the context.
Note that it always holds that $d(A,B)\le |A||B|$ where $|X|$ is the measure of a set $X$.
We would like to mention that
the density between $A$ and $B$ is often defined in a normalized way, i.e.~it is defined to be $\frac{d(A,B)}{|A||B|}$,
but this is not the case in this paper.

A partition of $[0,1]$ into measurable non-null sets $U_1,\ldots,U_k$ is said to be {\em $\varepsilon$-regular}
if it holds that
$$\left|d(A,B)-\sum_{i,j\in [k]}\frac{d(U_i,U_j)}{|U_i||U_j|}|U_i\cap A||U_j\cap B|\right|\le\varepsilon$$
for every two measurable subsets $A$ and $B$ of $[0,1]$.
The upper bound proof translates directly from graphs to graphons and
so we get that for every $\varepsilon$,
there exists $k_0\le 2^{O\left(\varepsilon^{-2}\right)}$ such that
every graphon has an $\varepsilon$-regular partition with at most $k_0$ parts.
Likewise,
the example of Conlon and Fox from~\cite{bib-conlon12+} can be used to obtain a step-graphon $W_{\varepsilon}$ such that
every $\varepsilon$-regular partition of $W_{\varepsilon}$ has at least $2^{\Omega\left(\varepsilon^{-2}\right)}$ parts.
However, the construction is probabilistic and the description of $W_{\varepsilon}$ is thus not explicit.
Based on this construction, we will define an explicit graphon $W_{\CF}^m$
which has similar properties as $W_{\varepsilon}$ for $\varepsilon\approx m^{-1/2}$.
In fact, a $W_{\CF}^m$-random graph of order $2^{\alpha m}$ for some $\alpha$ close to $0$
is the graph constructed by Conlon and Fox in~\cite{bib-conlon12+}.

Fix an integer $m$. The graphon $W_{\CF}^m$ is a step graphon that consists of $2^{m}$ parts of equal size.
Each of the parts is associated with a vector $u\in\{-1,+1\}^{m}$.
The part of the graphon $W_{\CF}^m$ corresponding to vectors $u$ and $u'$
is constantly equal to $\trunc{\frac{1}{2}+\frac{\scalar{u}{u'}}{4m^{1/2}}}$
where $\trunc{x}$ is equal to $x$ if $x\in[0,1]$, it is equal to $0$ if $x<0$ and to $1$ if $x>1$.
In other words, the operator $\trunc{\cdot}$ replaces values smaller than $0$ or larger than $1$ with $0$ and $1$, respectively.
Observe that $d([0,1],[0,1])=1/2$ by symmetry.
Using the Chernoff bound, one can show that the measure of the points $(x,y)$ with $0<W(x,y)<1$
is at least $1-2e^{-2}>1/2$.

It would be possible to relate the proof presented in~\cite{bib-conlon12+} to arguments on regular partitions of graphons.
However, the probabilistic nature of the construction would make this technical and obfuscate some simple ideas.
Because of this, and to keep the paper self-contained,
we decided to present a direct proof following the lines of the reasoning given in~\cite{bib-conlon12+}.

\begin{theorem}
\label{thm-conlon-fox}
If $m\ge 25$ and $\varepsilon<\frac{1}{2^{14}m^{1/2}}$,
then every $\varepsilon$-regular partition of the graphon $W_{\CF}^m$ has at least $2^{m/4}$ parts.
\end{theorem}

\begin{proof}
Fix an integer $m\ge 25$.
Let $V_i^-$ be the vertices of $W_{\CF}^m$ in the parts associated
with vectors $u$ whose $i$-th coordinate equals $-1$.
Similarly, $V_i^+$ are the vertices of $W_{\CF}^m$ in the parts associated
with vectors $u$ whose $i$-th coordinate equals $+1$.

Suppose that $W_{\CF}^m$ has an $\varepsilon$-regular partition $U_1,\ldots,U_k$
with $\varepsilon<2^{-14}m^{-1/2}$ and $k<2^{m/4}$.
We say that a part $U_t$ is {\em small} if $|U_t|\le 2^{-m/3}$.
Note that the sum of the measures of the small parts is at most $k\cdot 2^{-m/3}\le 1/2$.
For every $t\in [k]$, set
\begin{equation}
S_t=\sum_{i\in [m]}|U_t\cap V_i^-|\cdot|U_t\cap V_i^+|\;\mbox{.}
\label{eq-St}
\end{equation}
If $v\in\{-1,+1\}^{m}$, then the number of vectors $v'\in\{-1,+1\}^{m}$ such that
$v$ and $v'$ differ in at most $m/16$ coordinates is at most $2^{m-\frac{49}{128}m}\le 2^{2m/3-1}$
using the Chernoff bound and the fact that $m\ge 25$.
Hence, for every $v\in\{-1,+1\}^{m}$,
the measure of the vertices in the parts associated with vectors that differ from $v$ in at most $m/16$ coordinates is at most $2^{-m/3}/2$.
Consequently, if $U_t$ is not small, then
each vertex of $U_t$ contributes to the sum (\ref{eq-St}) by at least $\frac{m}{16}\left(|U_t|-2^{-m/3}/2\right)\ge m|U_t|/32$.
We conclude that $S_t\ge |U_t|^2m/32$ if $U_t$ is not small.

We say that the pair $(i,t)\in [m]\times [k]$ is {\em useful} if $\min\{|U_t\cap V_i^-|,|U_t\cap V_i^+|\}\ge |U_t|/64$.
Let $M_t$, $t\in [k]$, be the number of indices $i\in [m]$ such that the pair $(i,t)$ is useful.
Since each term in the sum (\ref{eq-St}) is at most $|U_t|^2/4$ and it is at most $|U_t|^2/64$ if $(i,t)$ is not useful,
it follows that $S_t\le |U_t|^2 (M_t/4+m/64)$.
We conclude that $M_t\ge m/16$ unless $U_t$ is small (recall that $S_t\ge |U_t|^2m/32$ if $U_t$ is not small).
Since the sum of the measures of the parts that are not small is at least $1/2$, we obtain that
\begin{equation}
\sum_{t\in [k]}M_t|U_t|\ge m/32\;\mbox{.}
\label{eq-Mt}
\end{equation}
In particular,
there exists $i_0\in [m]$ such that the sum of the measure of parts $U_t$ such that the pair $(i_0,t)$ is useful is at least $1/32$.
Fix such an index $i_0$ for the rest of the proof.

Let $A^-$ be any measurable subset of $V_{i_0}^-$ such that $|A^-\cap U_t|=|U_t|/64$ if $(i_0,t)$ is useful, and $|A^-\cap U_t|=0$ otherwise.
Similarly,
let $A^+$ be any measurable subset of $V_{i_0}^+$ such that $|A^+\cap U_t|=|U_t|/64$ if $(i_0,t)$ is useful, and $|A^+\cap U_t|=0$ otherwise.
Such sets $A^-$ and $A^+$ exist because $\min\{|U_t\cap V_{i_0}^-|,|U_t\cap V_{i_0}^+|\}\ge |U_t|/64$ for every $t$ such that $(i_0,t)$ is useful.
Note that the sets $A^-$ and $A^+$ have the same measure and
the choice of $i_0$ implies this measure is at least $1/2048=2^{-11}$.

Let $B=V_{i_0}^-$.
Since the partition $U_1,\ldots,U_k$ is $\varepsilon$-regular,
we get that
$$\left|d(A^+,B)-\sum_{i,j\in [k]}\frac{d(U_i,U_j)}{|U_i||U_j|}|U_i\cap A^+||U_j\cap B|\right|\le\varepsilon$$
and that
$$\left|d(A^-,B)-\sum_{i,j\in [k]}\frac{d(U_i,U_j)}{|U_i||U_j|}|U_i\cap A^-||U_j\cap B|\right|\le\varepsilon\mbox{.}$$
Since $|A^-\cap U_t|=|A^+\cap U_t|$ for every $t\in [k]$,
we infer that $|d(A^-,B)-d(A^+,B)|\le 2\varepsilon<2^{-13}m^{-1/2}$.
On the other hand, the choices of $A^-$, $A^+$ and $B$ imply that
$d(A^-,B)=(1/2+m^{-1/2}/4)|A^-||B|$ and $d(A^+,B)=(1/2-m^{-1/2}/4)|A^+||B|$.
In particular, it holds that
$$|d(A^-,B)-d(A^+,B)|=\frac{m^{-1/2}(|A^-|+|A^+|)|B|}{4}\ge 2^{-13}\cdot m^{-1/2}\;\mbox{.}$$
This contradicts the fact that $U_1,\ldots,U_k$ is an $\varepsilon$-regular partition of  $W_{\CF}^m$ 
with $\varepsilon<2^{-14}m^{-1/2}$ and $k<2^{m/4}$.
\end{proof}

At first sight, it might seem natural to consider the limit of the sequence of the graphons $W_{\CF}^m$, $m\in\NN$, as
a candidate for a graphon with no regular partitions with few parts.
It can be shown that the sequence of $W_{\CF}^m$, $m\in\NN$, is convergent,
however, its limit is (somewhat surprisingly) the graphon that maps every $(x,y)\in[0,1]^2$ to $1/2$,
which has an $\varepsilon$-regular partition with one part for every $\varepsilon$.

We finish this section with sketching a short argument explaining why Proposition~\ref{prop-lower-graphon} is true.
The proof of the weak regularity lemma in~\cite{bib-frieze99+} is based on iterative refinements of a partition of a graph(on),
at each step doubling the number of parts and increasing the ``mean square density'' by at least $\varepsilon^2$
until the partition becomes weakly $\varepsilon$-regular.
Suppose that there exists a graphon $W$ and $\varepsilon_i\to 0$ as in the proposition.
We can assume that $\varepsilon_{i+1}\le\varepsilon_i/2$ for every $i\in\NN$.
We start with a trivial partition into a single part and keep refining it until it becomes $\varepsilon_1$-regular;
let $k_1$ be the number of steps made. We then continue with refining until it becomes $\varepsilon_2$-regular and
let $k_2\ge k_1$ be the total number of steps made till this point.
We continue this procedure and define $k_i$, $i\ge 3$, in the analogous way.
Setting $k_0=0$, we conclude that the mean square density after $k_m$ steps is at least
$$\sum_{i=1}^{m} (k_i-k_{i-1})\varepsilon_i^2>\sum_{i=1}^{m} k_i \left(\varepsilon_i^2-\varepsilon_{i+1}^2\right)\ge \sum_{i=1}^{m}\frac{3}{4}k_i \varepsilon_i^2\;\mbox{.}$$
However, each $k_i$ must be at least $c\varepsilon_i^{-2}$ by the assumption of the proposition (otherwise,
$W$ would have a weak $\varepsilon_i$-regular partition with fewer than $2^{c\varepsilon_i^{-2}}$ parts).
Consequently, after $m=\left\lceil\frac{4}{3c}\right\rceil+1$ steps,
the mean square density exceeds one, which is impossible.

\section{Definition of the \svejk graphon}

We now define the \svejk graphon $W_S$.
We start with defining a tower function $t(n):\NN\to\NN$ as follows:
$$t(n)=\left\{\begin{array}{cl}
              1 & \mbox{if $n=0$, and}\\
	      2^{t(n-1)} & \mbox{otherwise.}
              \end{array}\right.$$
Note that $t(0)=1$, $t(1)=2$, $t(2)=2^2=4$, $t(3)=2^{2^2}=16$, $t(4)=2^{2^{2^2}}=65536$, etc.

The notation that we define next is summarized in Table~\ref{tab-brackets}.
For $x\in [0,1)$, we define $[x]_1$ to be the smallest integer $k$ such that $x<1-2^{-k}$.
In particular, $[x]_1=1$ iff $x\in [0,1/2)$, $[x]_1=2$ iff $x\in [1/2,3/4)$, $[x]_1=3$ iff $x\in [3/4,7/8)$, etc.
This allows us to view the interval $[0,1)$ as split into {\em segments} $[0,1/2)$, $[1/2,3/4)$, $[3/4,7/8)$, etc. and
$[x]_1$ is the index of the segment containing $x$ (numbered from one). We then define $\llbracket x\rrbracket_1$ to be $(x+2^{1-[x]_1}-1)\cdot 2^{[x]_1}$,
i.e.~$\llbracket x\rrbracket_1$ is the position of $x$ in the $[x]_1$-th segment if the $[x]_1$-th segment is scaled to the unit interval. 

\begin{table}[htbp]
\begin{center}
\begin{tabular}{|c|l|}
\hline
$[x]_1$ & index of the segment $[0,1/2)$, $[1/2,3/4)$, $[3/4,7/8)$, etc, containing $x$ \\
$[x]_2$ & index of the subsegment of the $[x]_1$-th segment containing $x$ (the segment is\\
 & divided into $t\left([x]_1\right)$ subsegments) \\
$[x]_3$ & index of the part of the $[x]_2$-th subsegment of the $[x]_1$-th segment containing\\
& $x$ (the subsegment is divided  into $t\left([x]_1\right)$ parts) \\
$[x]_{2,3}$ & $[x]_3 + t\left([x]_1\right)[x]_2$\\\hline
$\llbracket x\rrbracket_1$ & the position of $x$ in its segment scaled to be in $[0,1)$ \\
$\llbracket x\rrbracket_{1,i}^{01}$ & the $i$-th bit of the binary representation of $\llbracket x\rrbracket_1$\\
$\llbracket x\rrbracket_{1}^{\pm1}$ & a vector in $\{\pm1\}^{t\left([x]_1-1\right)}$ whose $i$-th coordinate is $2\llbracket x\rrbracket_{1,i}^{01}-1$\\
\hline
$[x]_{j,i}^{01}$ & the $i$-th bit of the binary representation of $[x]_j$\\ 
$[x]_{j}^{\pm1}$ & a vector in $\{\pm1\}^{t\left([x]_1-1\right)}$ whose $(i+1)$-th coordinate is $2[ x]_{j,i}^{01}-1$\\\hline
\end{tabular}
\end{center}
\caption{The notation used in the definition of the \svejk graphon.}
\label{tab-brackets}
\end{table}

Next, we let $[x]_2$ equal
$\left\lfloor \llbracket x\rrbracket_1\cdot t\left([x]_1\right)\right\rfloor$.
In other words, if the $[x]_1$-th segment of $[0,1)$ is divided into $t\left([x]_1\right)$ parts of the same length $2^{-[x]_1}/t\left([x]_1\right)$,
then $[x]_2$ is the index of the part containing $x$ if the parts are numbered from $0$.
We refer to these parts of the segments as {\em subsegments}. Analogously, we let $[x]_3$ be the index of the part containing $x$ when the $[x]_2$-th subsegment of the $[x]_1$-th segment is divided into $t\left([x]_1\right)$ parts of length $2^{-[x]_1}/t\left([x]_1\right)^2$, where the parts are numbered from $0$. Define $[x]_{2,3}$ to be $[x]_3 + t\left([x]_1\right)[x]_2$. Note that $[x]_{2,3}$ can also be viewed as the part containing $x$ when the $[x]_1$-th segment is divided into $t\left([x]_1\right)^2$ parts of length $2^{-[x]_1}/t\left([x]_1\right)^2$, and that $[x]_3$ is equal to $[x]_{2,3}$  reduced modulo $t\left([x]_1\right)$.

For $i\geq 1$, let $\llbracket x\rrbracket_{1,i}^{01}$ denote the $i$-th bit in the binary  representation of $\llbracket x\rrbracket_1$.
For example, if $\llbracket x\rrbracket_1=0.375=.011$ in binary,
then $\llbracket x\rrbracket_{1,1}^{01}=0$, $\llbracket x\rrbracket_{1,2}^{01}=\llbracket x\rrbracket_{1,3}^{01}=1$ and $\llbracket x\rrbracket_{1,i}^{01}=0$ for $i\ge 4$. We let $\llbracket x\rrbracket_{1}^{\pm1}$ denote the vector in $\{\pm1\}^{t\left([x]_1-1\right)}$ whose $i$-th coordinate is equal to $2\llbracket x\rrbracket_{1,i}^{01}-1$ (i.e.~$1$ is mapped to $+1$ and $0$ is mapped to $-1$).

For $j\in \{2,3\}$ and $i\geq 0$, let $[x]_{j,i}^{01}$ denote the $i$-th bit in the binary  representation of $[x]_j$. For example, if $[x]_j= 5 = 2^0 + 2^2$, then $[x]_{j,0}^{01}=[x]_{j,2}^{01}=1$, $[x]_{j,1}^{01}=0$ and $[ x]_{j,i}^{01}=0$ for $i\ge 3$. We let $[x]_{j}^{\pm1}$ denote the vector in $\{\pm1\}^{t\left([x]_1-1\right)}$ whose $(i+1)$-th coordinate is equal to $2[x]_{j,i}^{01}-1$.

\begin{figure}[htbp]
\begin{center}
\epsfbox{svejk.1}
\end{center}
\caption{The \svejk graphon.}
\end{figure}

The \svejk graphon $W_S$ has ten parts $A$, $B$, $C$, $D$, $E$, $F$, $G$, $P$, $Q$ and $R$.
For simplicity, we will define the graphon $W_S$ as a function $W_{13}$ from $[0,13)\times [0,13)$ to $[0,1]$, and
we set $W_S(x,y)=W_{13}(13x,13y)$.
All parts of $W_{13}$ except for $Q$ have measure one and we associate each of them with the unit interval $[0,1)$,
i.e.~we view the points of those parts as points in $[0,1)$. The remaining part $Q$ is associated with $[0,4)$.

\begin{table}[htbp]
\begin{center}
\begin{tabular}{|l|l|}
\hline
A pair $(x,y)$ belongs to & The value of $W_{13}(x,y)$ is $1$ if and only if\\
\hline
$A\times (A\cup B\cup \cdots \cup G)$ & $[x]_1=[y]_1$ \\
$B\times (B\cup E\cup F\cup G)$ & $[x]_1=[y]_1$ and $[x]_2=[y]_2$ \\
$B\times C$ & $t([x]_1-1)=[y]_1$ \\
$D\times C$ & $[x]_1=[y]_1+1$ \\
$D\times D$ & $[x]_1=[y]_1=1$ \\
$D\times G$ & $[x]_1=[y]_1$ and $\llbracket y\rrbracket_1\le \frac{1}{2} + \frac{\scalar{[x]_{2}^{\pm1}}{[x]_{3}^{\pm1}}}{4t\left([x]_1-1\right)^{1/2}}$ \\
$E\times C$ & $\llbracket x\rrbracket_{1,[y]_1}^{01}=1$ \\
$E\times D$ & $y\le 1-\llbracket x\rrbracket_1$ \\
$F\times C$ & $[y]_1\le t([x]_1-1)$, $\llbracket x\rrbracket_{1,[y]_1}^{01}=1$ and $\llbracket y\rrbracket_1\le t([x]_1)^{-1} 2^{[y]_1}$ \\
$F\times E$ & $[x]_1=[y]_1$ and $\llbracket y\rrbracket_1\le\frac{1}{2} - \frac{\scalar{[x]_{2}^{\pm1}}{[x]_{3}^{\pm1}}}{4t\left([x]_1-1\right)}$ \\
$F\times (D\cup F)$ or $G\times G$ & $[x]_{2,3}=[y]_{2,3}$ \\
$F\times G$ & $[x]_{3}=[y]_2$ \\
$G\times C$ & $[y]_1\le t([x]_1-1)$ and $\llbracket y\rrbracket_1\le t([x]_1)^{-1}2^{[y]_1}$ \\
$G\times E$ & $[x]_1=[y]_1$ and $1-\llbracket x\rrbracket_1\le \frac{1}{2} + t\left([x]_1-1\right)^{1/2}\left(\llbracket y\rrbracket_1-\frac{1}{2}\right)$ \\
$P\times (A\cup B\cup C\cup D)$ & $x\le y$ \\
$P\times (E\cup F\cup G\cup P)$ & $x\ge 1-y$ \\
\hline
\end{tabular}
\end{center}
\caption{The definition of the \svejk graphon on $(A\cup \cdots\cup G\cup P)^2$
         except on $C^2$, $E^2$, $B\times D$ and $D\times B$.}
\label{tab-svejk}
\end{table}

We will first define the values of the graphon $W_{13}$ between the pairs of the parts not involving $Q$ and $R$.
The graphon $W_{13}$ has values zero and one on $(A\cup \cdots \cup G\cup P)^2$
except on $C^2$, $E^2$, $B\times D$ and $D\times B$.
Table~\ref{tab-svejk} determines the values of $W_{13}$ in this zero-one case.
The values of $W_{13}$ on $C^2$, $E^2$ and $B\times D$ (by symmetry, this also determines the values on $D\times B$)
are defined as follows.
Note that the definition uses the graphon $W_{\CF}^m$ analyzed in Section~\ref{sec:weak},
which was defined just before the statement of Theorem~\ref{thm-conlon-fox}.

$$W_{13}(x,y)=\left\{\begin{array}{cl}
                2^{-2^{[x]_1-1}} & \mbox{if $[x]_1=[y]_1$, and} \\
		0 & \mbox{otherwise,}
                \end{array}\right.\mbox{ for $(x,y)\in C^2$,}$$

$$W_{13}(x,y)=\left\{\begin{array}{cl}
                W_{\CF}^{t([x]_1-1)}(\llbracket x\rrbracket_1,\llbracket y\rrbracket_1) & \mbox{if $[x]_1=[y]_1$, and} \\
		0 & \mbox{otherwise,}
                \end{array}\right.\mbox{ for $(x,y)\in E^2$, and}$$

$$W_{13}(x,y)=\left\{\begin{array}{cl}
                t([x]_1)^{-1} & \mbox{if $[x]_1=[y]_1$, and} \\
		0 & \mbox{otherwise,}
                \end{array}\right.\mbox{ for $(x,y)\in B\times D$.}$$

\noindent We have defined the values of the graphon $W_{13}$ on $(A\cup \cdots \cup G\cup P)^2$,
i.e.~between all pairs of its parts not involving $Q$ and $R$.

The part $Q$ is used to equalize degrees of the vertices in the parts $A,\ldots,G,P$,
i.e., to make the values
$$\frac{1}{13}\int_{[0,13)} W_{13}(x,y)\dd y\;\mbox{,}$$
to be the same for all $x$ from the same part;
see Section~\ref{sec-constraints} for further details on the degree of a vertex in a graphon.
If $x\in A\cup \cdots \cup G\cup P=\overline{Q\cup R}$ and $y\in Q$, then
$$W_{13}(x,y)=\frac{1}{4}\left(4-\int\limits_{\overline{Q\cup R}}W_{13}(x,z)\dd z\right)\;\mbox{.}$$
It is straightforward to verify that $W_{13}(x,y)\in [0,1]$ for every $(x,y)\in\overline{Q\cup R}\times Q$.

The part $R$ distinguishes the parts by vertex degrees. If $y\in R$, then
$$W_{13}(x,y)=\left\{\begin{array}{cl}
                1/8 & \mbox{if $x\in B$,}\\
                2/8 & \mbox{if $x\in C$,}\\
                3/8 & \mbox{if $x\in D$,}\\
                4/8 & \mbox{if $x\in E$,}\\
                5/8 & \mbox{if $x\in F$,}\\
                6/8 & \mbox{if $x\in G$,}\\
                7/8 & \mbox{if $x\in P$, and}\\
                0 & \mbox{otherwise.}
		\end{array}\right.$$
Finally, the graphon $W_{13}$ is equal to $1$ on $Q\times Q$.

The vertices in each of the ten parts of the \svejk graphon have the same degree (note that
$W(x_1,y)=W(x_2,y)$ for any two vertices $x_1,x_2\in Q$ and any $y\in [0,13)$).
This degree is given in Table~\ref{tab-degrees}.
We have not computed the degree of the vertices in the part $Q$ exactly
since it is enough to establish that this degree is larger than (and thus distinct from)
the degrees of the vertices in the other parts.

\begin{table}[htbp]
\begin{center}
\begin{tabular}{|l|c|c|c|c|c|c|c|c|c|c|}
\hline
Part & $A$ & $B$ & $C$ & $D$ & $E$ & $F$ & $G$ & $P$ & $Q$ & $R$ \\\hline & & & & & & & & & & \\[-1em] 
Degree & $\frac{32}{104}$ & $\frac{33}{104}$ & $\frac{34}{104}$ & $\frac{35}{104}$ & $\frac{36}{104}$ & $\frac{37}{104}$ & $\frac{38}{104}$ & $\frac{39}{104}$ & $\ge \frac{40}{104}$ & $\frac{28}{104}$ \\[.15em]\hline
\end{tabular}
\end{center}
\caption{The degrees of the vertices in each part of the \svejk graphon.}
\label{tab-degrees}
\end{table}

We finish this section by establishing that the \svejk graphon has no weak regular partitions with few parts.

\begin{proposition}
\label{thm-svejk-lower}
The \svejk graphon $W_S$ has no weak $\varepsilon$-regular partition with fewer than $2^{t(n)/4}$ parts
if $\varepsilon<\frac{1}{2^{24+2n}t(n)^{1/2}}$ and $n\ge 4$.
In particular, there exists a sequence of positive reals $\varepsilon_i$ tending to $0$ such that
every weak $\varepsilon_i$-regular partition of $W_S$ has at least $2^{\Omega\left(\varepsilon_i^{-2}/2^{5\log^*\varepsilon_i^{-2}}\right)}$ parts.
\end{proposition}

\begin{proof}
The graphon $W_S$ contains a copy of $W_{\CF}^{t(n)}$ scaled by $2^{-n-1}/13$ for every $n\in\NN$.
Note that a weak $\varepsilon$-regular partition of $W_S$ yields
a weak $\left(\varepsilon 2^{-2n}/676\right)$-regular partition of $W_{\CF}^{t(n)}$ with fewer or the same number of parts.
It follows that $W_S$ cannot have a weak $\varepsilon$-regular partition with fewer than $2^{t(n)/4}$ parts
for $\varepsilon<\frac{1}{676\cdot 2^{14+2n}\cdot t(n)^{1/2}}$ and $n\ge 4$ by Theorem~\ref{thm-conlon-fox}.

Setting $\varepsilon_i=\frac{1}{2^{25+2i}t(i)^{1/2}}$, we obtain the desired sequence of $\varepsilon_i$'s.
Note that
$$\lim_{i\to\infty}\frac{\log^*\varepsilon_i^{-2}}{i}=\lim_{i\to\infty}\frac{\log^*\left(2^{4i+50}t(i)\right)}{i}=1$$ and so
$\frac{t(i)}{4}\in\Omega\left(\varepsilon_i^{-2}/2^{5\log^*\varepsilon_i^{-2}}\right)$ as desired.
\end{proof}

\section{Constraints}
\label{sec-constraints}

The proof that the \svejk graphon is finitely forcible uses the notion of decorated constraints,
which was introduced in~\cite{bib-comp} and further developed in~\cite{bib-inf}.
We now present the notion following the lines of~\cite{bib-inf}.

A {\em constraint} is an equality between two density expressions
where a {\em density expression} is recursively defined as follows:
a real number or a graph $H$ are density expressions, and
if $D_1$ and $D_2$ are two density expressions, then the sum $D_1 + D_2$ and
the product $D_1\cdot D_2$ are also density expressions.
The value of the density expression for a graphon $W$
is the value obtained by substituting for each graph $H$ its density in $W$.

As observed in \cite{bib-comp},
if $W$ is the unique graphon (up to weak isomorphism) that satisfies a finite set $\CC$ of constraints, 
then it is finitely forcible.
In particular,
$W$ is the unique graphon with densities of subgraphs appearing in $\CC$ equal to their densities in $W$. 
Hence, a possible way of establishing that a graphon $W$ is finitely forcible is providing a finite set of constraints $\CC$ such that
the graphon $W$ is the unique graphon up to weak isomorphism that satisfies these constraints.

If $W$ is a graphon, then the points of $[0,1]$ can be viewed as vertices and
we can also speak of the {\em degree} of a vertex $x\in [0,1]$, defined as
$$\deg_W(x)=\int_{[0,1]} W(x,y)\dd y\;\mbox{.}$$
Note that the degree is well-defined for almost every vertex of $W$.
We will omit the subscript $W$ when the graphon is clear from the context. 
A graphon $W$ is {\em partitioned}
if there exist $k\in\NN$ and positive reals $a_1,\ldots,a_k$ summing to one and
distinct reals $d_1,\ldots,d_k$ between $0$ and $1$ such that
the set of vertices of $W$ with degree $d_i$ (referred to as a \emph{part} of the partitioned graphon) has measure $a_i$.
The following lemma was proven in~\cite{bib-comp}.

\begin{lemma}
\label{lm-partition}
Let $a_1,\ldots,a_k$ be positive real numbers summing to one and
let $d_1,\ldots,d_k$ be distinct reals between $0$ and $1$. There exists
a finite set of constraints $\CC$ such that any graphon $W$ satisfying $\CC$
is a partitioned graphon with parts of sizes $a_1,\ldots,a_k$ and degrees $d_1,\ldots,d_k$, and
any partitioned graphon with parts of sizes $a_1,\ldots,a_k$ and degrees $d_1,\ldots,d_k$
satisfies $\CC$.
\end{lemma}

We now introduce a stronger type of constraints, which was also used in~\cite{bib-inf,bib-comp}.
We will refer to the constraints introduced earlier as {\em ordinary constraints} if a distinction needs to be made.
Suppose that $W$ is a partitioned graphon with parts $A_i\subseteq [0,1]$, $i\in [k]$,
where the part $A_i$ has measure $a_i$ and it contains vertices of degrees $d_i$.
A {\em decorated graph} is a graph with some vertices distinguished as {\em roots} and
each vertex labeled with one of the parts $A_1,\ldots,A_k$;
the roots of a decorated graph come with a fixed order.
Two decorated graphs are {\em isomorphic}
if they have the same number of roots and there exists a bijection between their vertices
that is a graph isomorphism,
that maps roots to roots only while preserving their order, and
that preserves vertex labels.
Two decorated graphs are {\em compatible}
if the subgraphs induced by their roots are isomorphic (as decorated graphs).
A {\em decorated constraint} is a constraint
where all graphs appearing in the density expressions are compatible decorated graphs.
Note that decorated graphs and constraints are always defined
with a particular type of a partition of a graphon (i.e.~names of the parts) in mind.

We now define when a graphon $W$ satisfies a decorated constraint.
Fix a decorated constraint $C$.
Since all decorated graphs appearing in $C$ are compatible,
the roots of each of the decorated graphs appearing in $C$ induce the same decorated graph.
Let $H_0$ be this decorated graph and let $n$ be the number of its vertices;
note that all $n$ vertices of $H_0$ are roots.
The decorated constraint $C$ is {\em satisfied}
if the following holds for almost every $n$-tuple $x_1,\ldots,x_n$ of vertices of $H_0$ such that
$x_i$ belongs to the part that the $i$-th vertex of $H_0$ is labeled with,
$W(x_i,x_j)>0$ if $x_i$ and $x_j$ are adjacent in $H_0$, and
$W(x_i,x_j)<1$ if they are not adjacent:
the two sides of the constraint $C$ are equal when
each decorated graph $H$ is substituted with the probability that
a $W$-random graph with vertices corresponding to those of $H$ is the decorated graph $H$ 
conditioned on the root vertices being $x_1,\ldots,x_n$ and inducing the graph $H_0$ and
conditioned on each of the non-root vertices chosen from a part that it is labeled with.
Note that we do not allow any permutation of vertices in this definition,
i.e., the requirement is stronger than saying that the $W$-random graph is isomorphic to the decorated graph $H_0$.
A possible way of satisfying the constraint $C$ is that
the measure of the $n$-tuples $x_1,\ldots,x_n$ of vertices of $H_0$ with the properties given above is zero;
if this is the case, the constraint $C$ is said to be {\em null-satisfied}.

Before proceeding further, let us give an example of evaluating a decorated constraint.
Consider a partitioned graphon $W$ with two parts $A$ and $B$, each of measure one half, such that
$W$ is equal to one on $A\times A$, to zero on $B\times B$, and to one half on $A\times B$.
The graphon $W$ is depicted in Figure~\ref{fig-example1}.
Let $H$ be a decorated graph with two roots that are adjacent and both labeled with $A$ and
two non-root vertices $v_1$ and $v_2$ that are not adjacent, both labeled with $B$,
$v_1$ is adjacent to one of the roots and $v_2$ is adjacent to both roots.
The decorated graph $H$ is also depicted in Figure~\ref{fig-example1}.
If $H$ appears in a decorated constraint and its two roots are from the part $A$ of the graphon $W$,
then it will be substituted with the probability $1/16$ when evaluating the constraint with respect to $W$.
Note that if we allowed isomorphisms of decorated graphs when evaluating decorated constraints,
then this probability would be $2/16$ because the order in that the non-root vertices are chosen would be irrelevant.

\begin{figure}
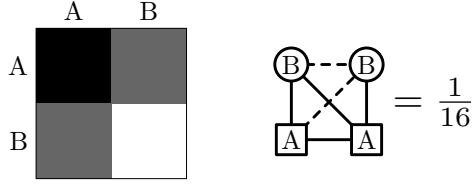

\begin{center}
\epsfbox{svejk.82} \hskip 10mm
\epsfbox{svejk.83}
\end{center}
\caption{An example of evaluating a decorated constraint.}
\label{fig-example1}
\end{figure}

The following lemma was proven in~\cite[Lemma 2]{bib-comp}, also see~\cite[Lemma 3]{bib-inf}.
\begin{lemma}
\label{lm-decorated}
Let $k\in \mathbb{N}$, let $a_1 , \ldots, a_k$ be positive real numbers summing to one, and
let $d_1, \ldots, d_k$ be distinct reals between zero and one.
If $W$ is a partitioned graphon with $k$ parts formed by vertices of degree $d_i$ and measure $a_i$ each,
then any decorated constraint can be expressed as a single ordinary constraint,
i.e.~$W$ satisfies the decorated constraint if and only if it satisfies the ordinary constraint.
\end{lemma}
\noindent By Lemma~\ref{lm-decorated},
we can equivalently work with (formally stronger) decorated constraints instead of ordinary constraints.

It is useful to fix some notation for visualizing decorated constraints.
We write the decorated constrains as expressions involving decorated graphs
where the roots are depicted by squares and non-root vertices by circles, and
each vertex is labeled with the name of the respective part of a graphon.
The solid lines connecting vertices correspond to the edges and dashed lines to the non-edges.
No connection between two vertices means that both edge or non-edge are allowed between the vertices,
i.e.~the picture should be interpreted as the sum of two graphs, one with an edge and with a non-edge.
If more than a single pair of vertices is not joined,
the picture should be interpreted as the multiple sum over all non-joined pairs of vertices,
which can lead to a sum containing several isomorphic copies of the same decorated graph.
An example is given in Figure~\ref{fig-example2}.
To avoid possible ambiguity, 
the drawings of the subgraph induced by the roots are identical for all decorated graphs in each constraint,
which makes clear which roots correspond to each other.

\begin{figure}
\begin{center}
\epsfbox{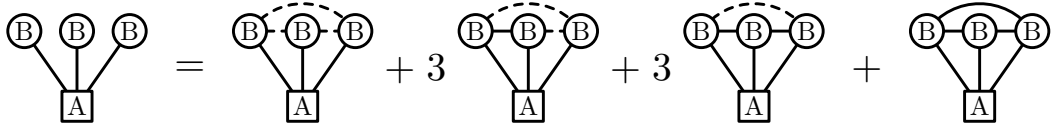}
\end{center}
\caption{An example of interpreting a drawing of a decorated graph with some unspecified adjacencies.}
\label{fig-example2}
\end{figure}

We remark that
if we allowed isomorphisms of decorated graphs in the computation of probabilities when evaluating decorated constraints,
then we would have to define the interpretation of the visualiziation of decorated constraints in a different way.
In particular, a decorated graph with some unspecified adjacencies between its vertices
would be replaced with the sum where each isomorphic copy appears with the coefficient one,
i.e., the coefficients $3$ would become $1$ in Figure~\ref{fig-example2}.
In this different setting, the visualization of the decorated constraints would actually be identical
with the following single exception: the first constraint on the last line in Figure~\ref{fig-subsegments-3},
where a coefficient to account for isomorphisms of the two graphs appearing in the constraint would have to be included.

We finish this section with the following lemma, which is an easy corollary of Lemma~\ref{lm-decorated}. In essence, it says that if a graphon $W_0$ can be finitely forced in its own right, then it can be forced on a single part of a partitioned graphon $W$ without affecting the structure of the other parts.

\begin{lemma}
\label{lm-subforcing}
Let $W_0$ be a finitely forcible graphon, let $a_1,\dots,a_k$ be positive reals summing to one and let $d_1,\dots,d_k$ be distinct reals between zero and one. Then there exists a finite set $\CC$ of decorated constraints such that a partitioned graphon $W$ with $k$ parts formed by vertices of degree $d_i$ and measure $a_i$ each satisfies $\CC$ if and only if the subgraphon of $W$
induced by the $m$-th part is weakly isomorphic to $W_0$. In other words, if the $m$-th part is denoted $A_m$, then $W$ satisfies $\CC$ if and only if
there exist measure preserving maps $\varphi:[0,a_m]\to A_m$ and $\varphi_0:[0,1]\to [0,1]$ such that
$W(\varphi(xa_m),\varphi(ya_m))=W_0(\varphi_0(x),\varphi_0(y))$ for almost every pair $(x,y)\in [0,1]^2$.
\end{lemma}

\begin{proof}
Let $H_1,\ldots,H_\ell$ and $d_1,\ldots,d_{\ell}$ be the subgraphs and their densities such that
$W_0$ is the unique graphon (up to weak isomorphism) with these densities.
The set $\CC$ is formed by $\ell$ decorated constraints:
the left side of the $i$-th constraint contains $H_i$ with all its vertices labelled by $A_m$ and
the right side is equal to $d_i$ divided by the number of automorphisms of $H_i$.
If the subgraphon of $W$ induced by $A_m$ is weakly isomorphic to $W_0$,
then clearly these constraints are satisfied.
On the other hand, since $W_0$ is forced by setting the densities of $H_i$ to $d_i$ for every $i\in [\ell]$,
the converse is true as well.
\end{proof}

\section{Finite forcibility of the \svejk graphon}
\label{sec-forcing}

Our final and longest section is devoted to proving that the \svejk graphon is finitely forcible.
We will prove this by exhibiting a finite set of constraints that the \svejk graphon satisfies and
showing that the \svejk graphon is the only graphon up to weak isomorphism that satisfies this set of constraints.
By Lemma~\ref{lm-partition}, there exists a finite set of constraints such that
any graphon that satisfies them is a partitioned graphon with ten parts of the sizes as in the \svejk graphon and
degrees of vertices in these parts as in Table~\ref{tab-degrees}.
Hence, we can work with decorated constraints with vertices labeled by the parts $A,\ldots,G$, $P$, $Q$ and $R$ (see Lemma~\ref{lm-decorated}).
We will use decorated constraints to enforce the structure of the graphon between pairs of its parts, one pair after another,
often building on the structure enforced by earlier constraints.
Table~\ref{tab-forcing} gives references to subsections where the structure between the particular pairs of parts is forced.

\begin{table}
\begin{center}
\begin{tabular}{|c|cccccccccc|}
\hline
& A & B & C & D & E & F & G & P & Q & R \\
\hline
A & \ref{sub-segments} & \ref{sub-segments} & \ref{sub-segments} & \ref{sub-segments} & \ref{sub-segments} & \ref{sub-segments} & \ref{sub-segments} & \ref{sub-coordinate} & \ref{sub-degrees} & \ref{sub-degrees} \\
B & \ref{sub-segments} & \ref{sub-subsegments} & \ref{sub-tower} & \ref{sub-tower} & \ref{sub-subsegments} & \ref{sub-subsegments} & \ref{sub-subsegments} & \ref{sub-coordinate} & \ref{sub-degrees} & \ref{sub-degrees} \\
C & \ref{sub-segments} & \ref{sub-tower} & \ref{sub-tower} & \ref{sub-segments} & \ref{sub-binary} & \ref{sub-binary} & \ref{sub-binary} & \ref{sub-coordinate}  & \ref{sub-degrees} & \ref{sub-degrees} \\
D & \ref{sub-segments} & \ref{sub-tower} & \ref{sub-segments} & \ref{sub-segments} & \ref{sub-binary} & \ref{sub-subsegments} & \ref{sub-dot} & \ref{sub-coordinate} & \ref{sub-degrees} & \ref{sub-degrees} \\
E & \ref{sub-segments} & \ref{sub-subsegments} & \ref{sub-binary} & \ref{sub-binary} & \ref{sub-main} & \ref{sub-dot} & \ref{sub-linear} & \ref{sub-coordinate} & \ref{sub-degrees} & \ref{sub-degrees} \\
F & \ref{sub-segments} & \ref{sub-subsegments} & \ref{sub-binary} & \ref{sub-subsegments} & \ref{sub-dot} & \ref{sub-subsegments} & \ref{sub-subsegments} & \ref{sub-coordinate} & \ref{sub-degrees} & \ref{sub-degrees} \\
G & \ref{sub-segments} & \ref{sub-subsegments} & \ref{sub-binary} & \ref{sub-dot} & \ref{sub-linear} & \ref{sub-subsegments} & \ref{sub-subsegments} & \ref{sub-coordinate} & \ref{sub-degrees} & \ref{sub-degrees} \\
P & \ref{sub-coordinate} & \ref{sub-coordinate} & \ref{sub-coordinate} & \ref{sub-coordinate} & \ref{sub-coordinate} & \ref{sub-coordinate} & \ref{sub-coordinate}  & \ref{sub-coordinate} & \ref{sub-degrees} & \ref{sub-degrees} \\
Q & \ref{sub-degrees} & \ref{sub-degrees} & \ref{sub-degrees} & \ref{sub-degrees} & \ref{sub-degrees} & \ref{sub-degrees} & \ref{sub-degrees} & \ref{sub-degrees} & \ref{sub-degrees} & \ref{sub-degrees} \\
R & \ref{sub-degrees} & \ref{sub-degrees} & \ref{sub-degrees} & \ref{sub-degrees} & \ref{sub-degrees} & \ref{sub-degrees} & \ref{sub-degrees} & \ref{sub-degrees} & \ref{sub-degrees} & \ref{sub-degrees} \\
\hline
\end{tabular}
\end{center}
\caption{The subsections of Section~\ref{sec-forcing} where the structure of the \svejk graphon between the corresponding pairs of parts is forced.}
\label{tab-forcing}
\end{table}

Fix a graphon $W_0$ that satisfies all the constraints presented in this section.
In particular, $W_0$ satisfies the constraints given by Lemma~\ref{lm-partition} and
it is a partitioned graphon with ten parts of the sizes as in the \svejk graphon and
degrees of vertices in these parts as in Table~\ref{tab-degrees}.
These ten parts of $W_0$ will be denoted by $A_0,\ldots,G_0$, $P_0$, $Q_0$ and $R_0$ in correspondence with the parts of the \svejk graphon.
We will show that $W_0$ is weakly isomorphic to the \svejk graphon.
To avoid cumbersome notation,
we will use some symbols, in particular, $\JJ$, $\xi$ and $\xi_x$, in a way specific to individual subsections.
The meaning will clearly be defined in the corresponding subsection, so no confusion could appear.

\subsection{Coordinate system}
\label{sub-coordinate}

The half-graphon $W_{\triangle}$, i.e.~the zero-one graphon defined by $W_{\triangle}(x,y)=1$ iff $x+y\ge 1$, is finitely forcible~\cite{bib-diaconis09+}, also see~\cite{bib-lovasz11+}.
By Lemma~\ref{lm-subforcing},
there exists a finite set of decorated constraints such that $W_0$ satisfies these constraints if and only if the subgraphon induced by the part $P_0$ is weakly isomorphic to the half-graphon $W_{\triangle}$. We insist that $W_0$ satisfies these constraints.

Let $X\in \{A,\ldots,G,P\}$. We use the symbol $X_0$ to refer to the corresponding element of $\{A_0,\ldots,G_0,P_0\}$. 
By the Monotone Reordering Theorem (see~\cite[Proposition A.19]{bib-lovasz-book} for more details),
there exist measure preserving maps $\varphi_X:X_0\to [0,|X_0|)$ and non-decreasing functions $f_X:[0,|X_0|)\to [0,1)$, 
such that
$$f_X(\varphi_X(x)) = 13 \int\limits_{P_0} W_0(x,z) \dd z$$
for almost every $x\in X_0$. Since we already know that the subgraphon of $W_0$ induced by $P$ is weakly isomorphic to $W_{\triangle}$, we must have $W_0(x,y)=1$ for almost every pair $(x,y)\in P_0^2$ with $f_P(\varphi_P(x))+f_P(\varphi_P(y))\ge 1$,
$W_0(x,y)=0$ for almost every pair $(x,y)\in P_0^2$ with $f_P(\varphi_P(x))+f_P(\varphi_P(y))<1$, and
$f_P(z)=13z$ for almost every $z\in [0,1/13)$.

Set $g_X(x)=f_X(\varphi_X(x))$ for $x\in X_0$ and $X\in\{A,\ldots,G,P\}$.
For completeness, let $g_Q$ and $g_R$ be any measurable maps from $Q_0$ and $R_0$ to $Q\cong [0,4)$ and $R\cong [0,1)$ such that
for any measurable subset $Z$ of $[0,4)$ and $[0,1)$, respectively, we have $|g_Q^{-1}(Z)| = |Z|/13$ and $|g_R^{-1}(Z)| = |Z|/13$, respectively. 
Each $g_X$ can be viewed as a map from $X_0$ to the part $X$ of the graphon $W_{13}$.
The maps $g_A,\ldots,g_G$, $g_P$, $g_Q$ and $g_R$ constitute a map $g$ from the vertices of $W_0$ to the vertices of $W_{13}$ and so to those of $W_S$.

We will argue that the map $g$ as a map from the vertices $W_0$ to the vertices of $W_S$ is measure preserving and
we will show that $W_0(x,y)=W_S(g(x),g(y))$ for almost every pair $(x,y)\in [0,1]^2$.
This will prove that the graphons $W_0$ and $W_S$ are weakly isomorphic.
So far, we have established that $W_0(p,p')=W_S(g(p),g(p'))$ for almost every pair $(p,p')\in P_0^2$ and
that the map $g$ is measure preserving when restricted to $P_0\cup Q_0\cup R_0$.

\begin{figure}[htbp]
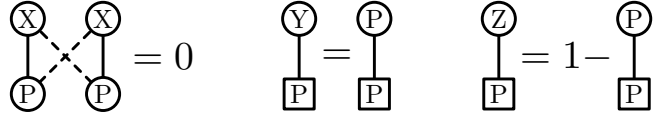

\begin{center}
\epsfbox{svejk.4} \hskip 10mm
\epsfbox{svejk.2} \hskip 10mm
\epsfbox{svejk.3}
\end{center}
\caption{Decorated constraints used in Subsection~\ref{sub-coordinate}
         where $X\in\{A,B,C,D,E,F,G\}$, $Y\in\{E,F,G\}$ and $Z\in\{A,B,C,D\}$.}
\label{fig-coordinate}
\end{figure}

Let us consider the decorated constraints depicted in Figure~\ref{fig-coordinate}.
Let $N_X(x)=\{y\in X_0\;|\;W_0(x,y)=1\}$ for $x\in P_0$ and $X\in\{A,\ldots,G\}$.
The first constraint implies that the graphon $W_0$ is zero-one valued almost everywhere on $P_0\times(A_0\cup\cdots\cup G_0)$ and
that $N_X(x)\setminus N_X(x')$ or $N_X(x')\setminus N_X(x)$ has measure zero for almost every pair $(x,x')\in P_0^2$ and for every $X\in\{A,\ldots,G\}$.
The second constraint in Figure~\ref{fig-coordinate} implies for $Y\in\{E,F,G\}$ that
the measure of $N_Y(x)$ is $g_P(x)$ for almost every $x\in P_0$.
Hence, it must hold that
$f_Y(y)=13y$ for $y\in [0,1/13)$ and
$W_0(x,y)=1$ for almost every $(x,y)\in P_0\times Y_0$ with $g_P(x)+g_Y(y)\ge 1$.
Similarly,
the third constraint in Figure~\ref{fig-coordinate} implies for $Z\in\{A,B,C,D\}$ that
the measure of $N_Z(x)$ is $1-g_P(x)$ for almost every $x\in P_0$.
Consequently, it holds that
$f_Z(y)=13y$ for $y\in [0,1/13)$ and
$W_0(x,y)=1$ for almost every $(x,y)\in P_0\times Z_0$ with $g_P(x)\ge g_Z(x)$.
We conclude that $g$ is a measure preserving map on the whole domain and
$W_0(x,y)=W_S(g(x),g(y))$ for almost every pair $(x,y)\in P_0\times\overline{(Q_0\cup R_0)}$.

The values of the functions $g_A,\ldots,g_G$ can be understood to be the coordinates of the vertices in $A_0,\ldots,G_0$, respectively, and
the coordinate of a vertex $x\in A_0\cup\cdots\cup G_0$ is
$$g_X(x)=13\int\limits_{P_0}W_0(x,z)\dd z$$
for $X\in\{A,\dots,G\}$.
This integral is easily expressible as a decorated density expression
since it is just the relative edge density (degree) of $x$ to $P_0$.
This view allows us to speak about segments and subsegments of the parts $A_0,\ldots,G_0$.
The $k$-th segment of $X_0$, $X\in\{A,\ldots,G\}$, is formed by those $x\in X_0$ such that $[g_X(x)]_1=k$.
Analogously, the values $[g_X(x)]_2$ determine the subsegments.

\subsection{Segmenting}
\label{sub-segments}

We now force that the parts $A_0,\ldots,G_0$ of $W_0$ are split into segments as in $W_S$.
We also force the structure to recognize the first segment through the clique inside $D_0^2$ and
to have the ``successor'' relation on the segments through the structure inside $C_0\times D_0$.
All three of these aims will be achieved by the decorated constraints given in Figure~\ref{fig-segments};
the arguments follows those given in~\cite{bib-inf,bib-comp}.

\begin{figure}[htbp]
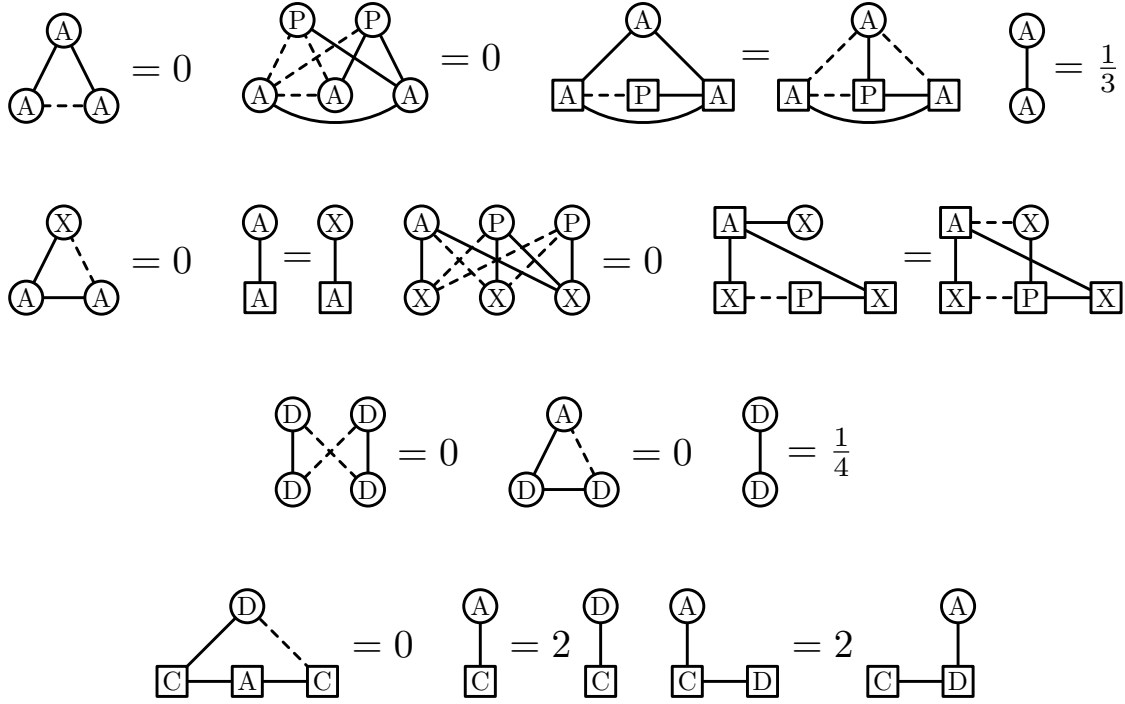

\begin{center}
\epsfbox{svejk.5} \hskip 5mm
\epsfbox{svejk.6} \hskip 5mm
\epsfbox{svejk.7} \hskip 5mm
\epsfbox{svejk.8} \vskip 10mm
\epsfbox{svejk.9} \hskip 5mm
\epsfbox{svejk.10} \hskip 5mm
\epsfbox{svejk.11} \hskip 5mm
\epsfbox{svejk.12} \vskip 10mm
\epsfbox{svejk.13} \hskip 5mm
\epsfbox{svejk.14} \hskip 5mm
\epsfbox{svejk.15} \vskip 10mm
\epsfbox{svejk.16} \hskip 5mm
\epsfbox{svejk.17} \hskip 5mm
\epsfbox{svejk.18} 
\end{center}
\caption{Decorated constraints used in Subsection~\ref{sub-segments}
         where $X\in\{B,C,D,E,F,G\}$.}
	 \label{fig-segments}
\end{figure}

The four constraints on the first line in Figure~\ref{fig-segments} force the structure on $A_0^2$.
The first constraint implies that there exists a set $\JJ$ of disjoint measurable subsets of $A_0$ such that
for almost every $x\in A_0$, there exists $J\in\JJ$ such that $W_0(x,y)=1$ for almost every $x,y\in J$ and
$W_0(x,y)=0$ for almost every $x\in J$ and $y\not\in J$. 
Hence, $W_0(x,y)=1$ for almost every $(x,y)\in\bigcup_{J\in\JJ}J^2$ and
$W_0(x,y)=0$ for almost every $(x,y)\in A_0^2\setminus\bigcup_{J\in\JJ}J^2$.

We claim that the second constraint together with the structure on $A_0\times P_0$ yields that
for every set $J\in\JJ$ there exists an open interval $J'\subseteq [0,1)$ such that $J$ and $g_A^{-1}(J')$ differ on a set of measure zero.
Note that such an open interval $J'$ might be empty.
Since we use an argument of this kind for the first time in this paper, we give more details.
If one of the (non-null) sets $J$ did not have the property, then a random sampling of three points $x,x',x''\in J\subseteq A_0$ with $g_A(x)<g_A(x')<g_A(x'')$ would satisfy
$W_0(x,x')=0$ and $W_0(x,x'')=1$ with positive probability.
For such three points, the probability of sampling the additional two points from $P_0$ is $g_A(x')-g_A(x)$ and $g_A(x'')-g_A(x')$ and
the triples of points $x,x',x''$ such that
the differences $g_A(x')-g_A(x)$ and $g_A(x'')-g_A(x')$ would be bounded away from zero have positive measure. 
Let $\JJ'$ be the set of open intervals $J'\subseteq [0,1)$ such that $g_A^{-1}\left(J'\right)$ and $J$ differ on a set of measure zero for some $J\in\JJ$.
Since the sets in $\JJ$ are disjoint and the sets in $\JJ'$ are open, the intervals of $\JJ'$ are also disjoint.

The third constraint implies that if $x,x'\in g_A^{-1}\left(J'\right)$ for some $J'\in\JJ'$,
then the measure of the interval $J'$, assuming it is non-empty, and the measure of the interval $\left(\sup J',1\right)$ are the same.
Again, we provide a detailed justification since we use an argument of this kind for the first time.
Almost every choice of the three roots $x\in A_0$, $x'\in P_0$ and $x''\in A_0$ (the order follows that in the figure)
satisfies that $g_A(x)<g_P\left(x'\right)<g_A\left(x''\right)$ (because of the non-edge between $x$ and $x'$ and the edge between $x'$ and $x''$) and that
there exists $J'\in\JJ'$ such that $x,x''\in g_A^{-1}\left(J'\right)$ (because of the edge between $x$ and $x''$).
The left side is then equal to the measure of $g_A^{-1}\left(J'\right)$, which is the measure of $J'$.
The right side is equal to the measure of those $z\in A_0$ such that $z\not\in g_A^{-1}\left(J'\right)$ and $g_A(z)>g_P\left(x'\right)$.
Hence, the right side is equal to $1-\sup J'$.
Since this holds for almost every triple $x$, $x'$ and $x''$,
we conclude that the measure of each non-empty interval $J'\in\JJ'$ is $1-\sup J'$.
Consequently,
each non-empty interval $J'$ must be of the form $(1-2\alpha,1-\alpha)$ for some $\alpha\in [0,1)$.
Since the intervals of $\JJ'$ are disjoint,
there can only be a finite number of intervals to the left of each interval of $\JJ'$.
This implies that the set $\JJ'$ is countable.

Finally, the last constraint on the first line yields that
$$\int\limits_{A_0^2}W(x,y)\dd x\dd y=\sum_{J'\in\JJ'}\left(\sup J'-\inf J'\right)^2=\frac{1}{3}\;\mbox{.}$$
However, this equality can hold only if the intervals contained in $\JJ'$ are exactly the intervals $\left(1-2^{1-k},1-2^{-k}\right)$, $k\in\NN$.
We conclude that $W_0(x,y)$ and $W_S(g(x),g(y))$ are equal for almost every pair $(x,y)\in A_0^2$.

We now analyze the four constraints on the second line in Figure~\ref{fig-segments}.
Fix $X\in\{B,\ldots,G\}$.
The first constraint implies that for every $J\in\JJ$, there exists $Z_J\subseteq X_0$ such that
$W_0(x,y)=1$ for almost every $(x,y)\in J\times Z_J$ and $W_0(x,y)=0$ for almost every $(x,y)\in J\times (X_0\setminus Z_J)$.
The second constraint yields that the measure of $Z_J$ is the same as the measure of $J$.
The third constraint implies that there exists an open interval $Z'_J$ such that $Z_J$ and $g_X^{-1}\left(Z'_J\right)$ differ on a set of measure zero.
Finally, the last constraint on the second line yields that each of the intervals $Z'_J$ is of the form $(1-2\alpha,1-\alpha)$ for some $\alpha\in [0,1)$.
Since the length of $Z'_J$ is the same as the measure of $J$, we conclude that
if $J=g_A^{-1}((1-2^{1-k},1-2^{-k}))$, then $Z'_J=(1-2^{1-k},1-2^{-k})$.
Hence, $W_0(x,y)$ and $W_S(g(x),g(y))$ are equal for almost every pair $(x,y)\in A_0\times X_0$, $X\in\{B,\ldots,G\}$.

Let us turn our attention to the three constraints on the third line in Figure~\ref{fig-segments}.
The first constraint implies that there exists a subset $Z_D$ of $D_0$ such that
$W_0(x,y)=1$ for almost every $(x,y)\in Z_D^2$ and
$W_0(x,y)=0$ for almost every $(x,y)\in D_0^2\setminus Z_D^2$.
The second constraint yields that $Z_D$ is a subset of $Z_J$ for some $J\in\JJ$.
Finally, the third constraint says that the square of the measure of $Z_D$ is $1/4$, i.e.~the measure of $Z_D$ is $1/2$.
However, this is only possible if $Z_D$ and $g_D^{-1}((0,1/2))$ differ on a set of measure zero.
We conclude that $W_0(x,y)$ and $W_S(g(x),g(y))$ are equal for almost every pair $(x,y)\in D_0^2$.

It remains to analyze the three constraints on the last line in Figure~\ref{fig-segments}.
The first constraint yields that for every $k\in\NN$,
there exists $Z_k\subseteq D_0$ such that $W(x,y)=1$ for almost every $(x,y)\in g_C^{-1}((1-2^{1-k},1-2^{-k}))\times Z_k$ and
$W(x,y)=0$ for almost every $(x,y)\in g_C^{-1}((1-2^{1-k},1-2^{-k}))\times (D_0\setminus Z_k)$.
The second constraint yields that the measure of $Z_k$ is $2^{-k-1}$, and
the third constraint that $Z_k$ is a subset of $g_D^{-1}((1-2^{-k},1-2^{-k-1}))$ except for a set of measure zero.
Hence, $Z_k$ and $g_D^{-1}((1-2^{-k},1-2^{-k-1}))$ differ on a set of measure zero, and
we can conclude that $W_0(x,y)$ and $W_S(g(x),g(y))$ are equal for almost every pair $(x,y)\in C_0\times D_0$.

\subsection{Tower function}
\label{sub-tower}

In this subsection, we will force a representation of the tower inside $B_0\times D_0$.
This is achieved using the constraints depicted in Figure~\ref{fig-tower}.
Before analyzing these constraints,
we give an analytic observation based on~\cite[proof of Lemma 3.3]{bib-lovasz11+}.

\begin{lemma}
\label{lm-square}
Let $F:[0,1)^2\to [0,1)$ be a measurable function.
If
$$\int\limits_{[0,1)} F(x,z)F(y,z)\dd z=\xi$$
for almost every $(x,y)\in [0,1)^2$, then
$$\int\limits_{[0,1)} F(x,z)^2\dd z=\xi$$
for almost every $x\in [0,1)$.
\end{lemma}

\begin{figure}[htbp]
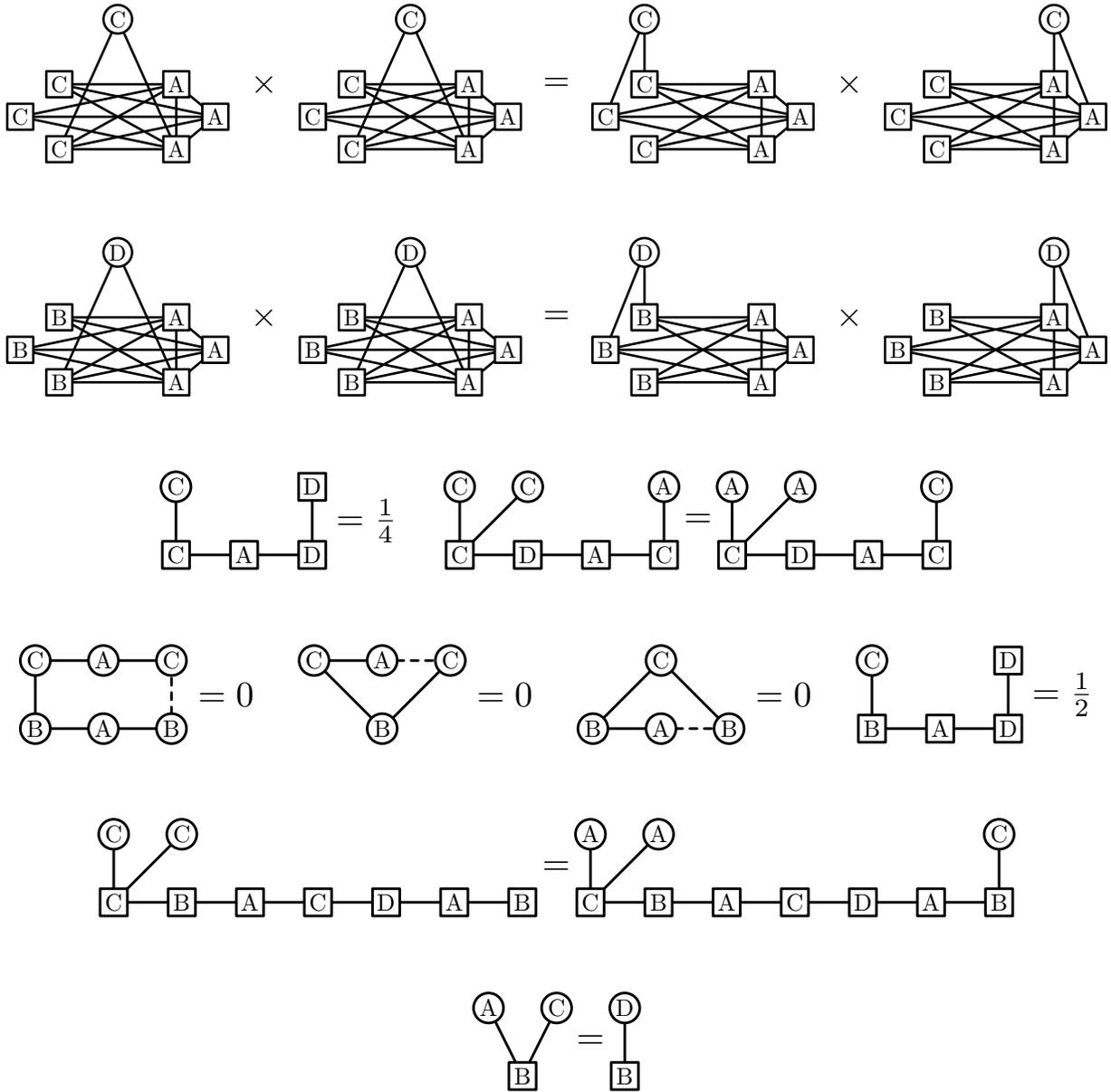

\begin{center}
\epsfxsize 16.2cm
\epsfbox{svejk.19} \vskip 10mm
\epsfxsize 16.2cm
\epsfbox{svejk.27} \vskip 10mm
\epsfbox{svejk.20} \hskip 5mm
\epsfbox{svejk.21} \vskip 10mm
\epsfbox{svejk.22} \hskip 5mm
\epsfbox{svejk.23} \hskip 5mm
\epsfbox{svejk.24} \hskip 5mm
\epsfbox{svejk.25} \vskip 10mm
\epsfbox{svejk.26} \vskip 10mm
\epsfbox{svejk.28}
\end{center}
\caption{Decorated constraints used in Subsection~\ref{sub-tower}.}
\label{fig-tower}
\end{figure}

The constraint on the first line in Figure~\ref{fig-tower} yields that
$$\left(\int\limits_{C_0} W_0(x,z)W_0(y,z)\dd z\right)^2=\left(\int\limits_{C_0} W_0(x',z)W_0(x'',z) \dd z\right) \left(\int\limits_{C_0} W_0(y',z)W_0(y'',z) \dd z\right)$$
for almost every $x,x',x''\in C_0$ and $y,y',y''\in A_0$ all in the same segment (this is implied by the presence of the edges between the roots). 
By Lemma~\ref{lm-square}, we get that
$$\left(\int\limits_{C_0} W_0(x,z)W_0(y,z)\dd z\right)^2=\left(\int\limits_{C_0} W_0(x',z)^2 \dd z\right) \left(\int\limits_{C_0} W_0(y',z)^2 \dd z\right)$$
for almost every $x,x'\in C_0$ and $y,y'\in A_0$ such that $[g_C(x)]_1=[g_A(y)]_1$.
Since the equality holds for almost every $x,x'\in C_0$ and $y,y'\in A_0$, it actually holds that
$$\left(\int\limits_{C_0} W_0(x,z)W_0(y,z)\dd z\right)^2=\left(\int\limits_{C_0} W_0(x,z)^2 \dd z\right) \left(\int\limits_{C_0} W_0(y,z)^2 \dd z\right)$$
for almost every $x\in C_0$ and $y\in A_0$ such that $[g_C(x)]_1=[g_A(y)]_1$.
The Cauchy-Schwarz Inequality yields that for every $k\in\NN$ there exists $\xi_k\in\RR$ such that $W_0(x,z)=\xi_k\cdot W_0(y,z)$ 
for almost every $x\in C_0$, $y\in A_0$ and $z\in C_0$ with $[g_C(x)]_1=[g_A(y)]_1=k$.
Hence, $W_0(x,z)=\xi_k$ for almost every $x\in C_0$ and $z\in C_0$ with $[g_C(x)]_1=[g_C(z)]_1=k$ and
$W_0(x,z)=0$ for almost every $x\in C_0$ and $z\in C_0$ with $[g_C(x)]_1\not=[g_C(z)]_1$.
Along the same line, the constraint on the second line implies that
for every $k\in\NN$ there exists $\xi'_k\in\RR$ such that $W_0(x,z)=\xi'_k\cdot W_0(y,z)$
for almost every $x\in B_0$, $y\in A_0$ and $z\in D_0$ with $[g_B(x)]_1=[g_A(y)]_1=k$.
Consequently, $W_0(x,z)=\xi'_k$ for almost every $x\in D_0$ and $z\in B_0$ with $[g_D(x)]_1=[g_B(z)]_1=k$ and
$W_0(x,z)=0$ for almost every $x\in D_0$ and $z\in B_0$ with $[g_D(x)]_1\not=[g_B(z)]_1$.

Almost every choice of the roots in the first constraint on the third line satisfies that all the roots belong to the same segment and
this segment must be the first segment because of the edge between the two roots from $D_0$.
Hence, this constraint implies that $\xi_1|g_C^{-1}((0,1/2))|=1/4$, i.e.~$\xi_1=1/2$ as desired.

Let us now look at the second constraint. 
Almost every choice of the roots satisfies that if the right root from $C_0$ is in the $k$-th segment,
then the roots from $A_0$ and $D_0$ are also in the $k$-th segment and the left root from $C_0$ is in the $(k-1)$-th segment.
Since for every $k$ the choice of such roots has positive probability,
the constraint implies that the following holds for every $k\in\NN$:
$$\left(\xi_k\left|g_C^{-1}\left(\left(1-2^{-k+1},1-2^{-k}\right)\right)\right|\right)^2\left|g_A^{-1}\left(\left(1-2^{-k},1-2^{-k-1}\right)\right)\right|=$$
$$\left(\left|g_A^{-1}\left(\left(1-2^{-k+1},1-2^{-k}\right)\right)\right|\right)^2\xi_{k+1}\left|g_C^{-1}\left(\left(1-2^{-k},1-2^{-k-1}\right)\right)\right|\;\mbox{.}$$
Hence, it holds that $\xi_{k+1}=\xi_k^2=2^{-2^{k-1}}$.
We conclude that $W_0(x,y)$ and $W_S(g(x),g(y))$ are equal for almost every pair $(x,y)\in C_0\times C_0$.

The first three constraints on the fourth line in Figure~\ref{fig-tower} yield that for every $k\in\NN$
either $W(x,y)=0$ for almost every $x\in B_0$ in the $k$-th segment and almost every $y\in C_0$ or
there exists $m_k$ such that $W(x,y)=1$ for almost every $x\in B_0$ in the $k$-th segment and almost every $y\in C_0$ in the $m_k$-th segment and
$W(x,y)=0$ for almost every $y\in C_0$ not in the $m_k$-th segment.
The last constraint on the fourth line yields that $m_1=1$.

We now show that the constraint on the fifth line implies that $m_k$ exists and $m_k=t(k-1)$ for every $k\in\NN$.
For almost every choice of the roots in the constraint on the fifth line,
if the right root from $B_0$ belongs to the $k$-th segment,
the left root from $B_0$ belongs to the $(k-1)$-th segment and
the left root from $C_0$ belongs to the $m_{k-1}$-th segment.
We derive that this constraint implies that
$$\left(2^{-2^{m_{k-1}-1}}\cdot 2^{-m_{k-1}}\right)^2=
  \left(2^{-m_{k-1}}\right)^2\cdot 2^{-m_k}\;\mbox{.}$$
We conclude that $m_k=2^{m_{k-1}}$ and so $m_k=t(k-1)$.
Consequently, $W_0(x,y)$ and $W_S(g(x),g(y))$ are equal for almost every pair $(x,y)\in B_0\times C_0$.

The constraint on the last line in Figure~\ref{fig-tower} yields by considering a choice of the root in the $k$-th segment of $B_0$ that
$$2^{-k}\cdot 2^{-m_k} = \xi'_k\cdot 2^{-k} \;\mbox{.}$$
We conclude that $\xi'_k=2^{-m_k}=2^{-t(k-1)}=t(k)^{-1}$.
Hence, $W_0(x,y)$ and $W_S(g(x),g(y))$ are equal for almost every pair $(x,y)\in B_0\times D_0$.

\subsection{Subsegmenting}
\label{sub-subsegments}

We now force the parts of the graphon that further structure the segments, e.g., provide the structure of subsegments;
these are the parts $B_0\times (B_0\cup E_0\cup F_0\cup G_0)$, $F_0\times (D_0\cup F_0\cup G_0)$ and $G_0^2$.
Some of the arguments are analogous to those presented earlier in Subsection~\ref{sub-segments}.
In analogy to Subsection~\ref{sub-segments},
the first two constraints in Figure~\ref{fig-subsegments-1} for $X=B$ yield that
there exists a set $\JJ$ of disjoint open subintervals of $[0,1)$ such that
$W_0(x,y)=1$ for almost every $(x,y)\in g_B^{-1}(J)^2$ for some $J\in\JJ$ and $W_0(x,y)=0$ for almost all other pairs $(x,y)\in B_0^2$.
The third constraint implies that
every set $g_B^{-1}(J)$ is a subset of $g_A^{-1}((1-2^{-k+1},1-2^{-k}))$ except for a set of measure zero for some $k\in\NN$.
Hence, each interval $J\in\JJ$ is a subinterval of $(1-2^{-k+1},1-2^{-k})$ for some $k\in\NN$.
The fourth constraint with $X=B$ yields that the length of each interval $J$ is $2^{-k}t(k)^{-1}$.
Finally, the first constraint on the second line can hold
only if each interval $(1-2^{-k+1},1-2^{-k})$ contains $t(k)$ such intervals $J$.
We conclude that $W_0(x,y)$ and $W_S(g(x),g(y))$ are equal for almost every pair $(x,y)\in B_0^2$.

\begin{figure}[htbp]
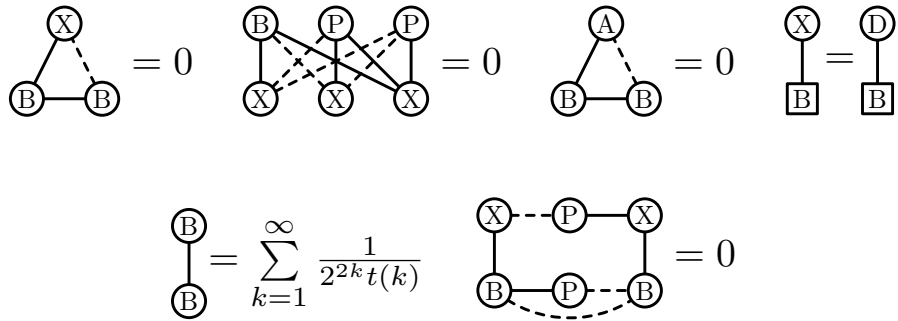

\begin{center}
\epsfbox{svejk.32} \hskip 5mm
\epsfbox{svejk.29} \hskip 5mm
\epsfbox{svejk.80} \hskip 5mm
\epsfbox{svejk.30} \vskip 10mm
\epsfbox{svejk.31} \hskip 5mm
\epsfbox{svejk.33}
\end{center}
\caption{The first set of decorated constraints used in Subsection~\ref{sub-subsegments}
         where $X\in\{B,E,F,G\}$.}
	 \label{fig-subsegments-1}
\end{figure}

For $X\in\{E,F,G\}$, the first, second and fourth constraints on the first line
yield that for each $J\in\JJ$ there exists an open interval $J'$ of the same length as $J$ such that
$W_0(x,y)=1$ for almost every $(x,y)\in J\times J'$ and
$W_0(x,y)=0$ for almost every $(x,y)\in J\times \left(X_0\setminus J'\right)$.
The last constraint on the second line in Figure~\ref{fig-subsegments-1} gives that
the intervals $J'$ follow in the same order as the intervals $J$.
Hence, $W_0(x,y)$ and $W_S(g(x),g(y))$ are equal for almost every pair $(x,y)\in B_0\times (E_0\cup F_0\cup G_0)$.

\begin{figure}[htbp]
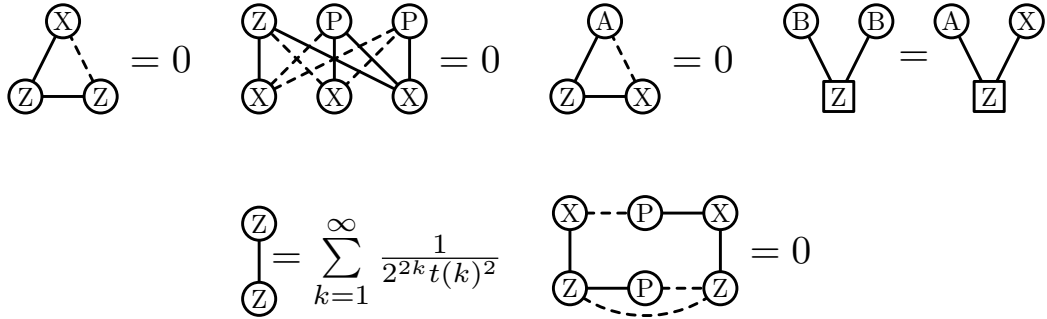

\begin{center}
\epsfbox{svejk.37} \hskip 5mm
\epsfbox{svejk.34} \hskip 5mm
\epsfbox{svejk.81} \hskip 5mm
\epsfbox{svejk.35} \vskip 10mm
\epsfbox{svejk.36} \hskip 5mm
\epsfbox{svejk.38}
\end{center}
\caption{The second set of decorated constraints used in Subsection~\ref{sub-subsegments}
         where $(Z,X)\in\{(F,F), (F,D), (G,G) \}$.}
	 \label{fig-subsegments-2}
\end{figure}

The set of constraints in Figure~\ref{fig-subsegments-2} is analogous to those in Figure~\ref{fig-subsegments-1}.
The main difference is the fourth constraint, which forces that
if an interval $J$ from the set $\JJ'$ corresponding to $F_0^2$ or $G_0^2$
is a subinterval of an interval $(1-2^{-k+1},1-2^{-k})$,
then $\left(2^{-k}t(k)^{-1}\right)^2=2^{-k}\cdot|J|$.
Hence, the length of such an interval $J$ must be $2^{-k}t(k)^{-2}$.
The first constraint on the second line then forces that the interval $(1-2^{-k+1},1-2^{-k})$ must contain $t(k)^2$ such intervals $J$ and
the order of the corresponding pairs of intervals is forced by the last constraint.
We can now conclude that
$W_0(x,y)$ and $W_S(g(x),g(y))$ are equal for almost every pair $(x,y)\in F_0^2\cup G_0^2\cup (F_0\times D_0)$.

\begin{figure}[htbp]
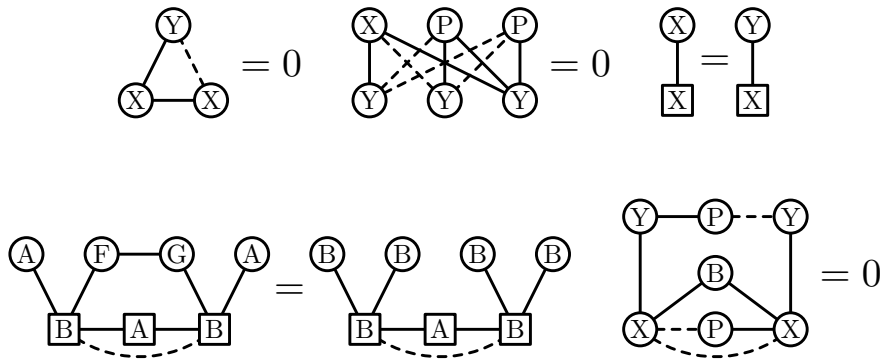

\begin{center}
\epsfbox{svejk.40} \hskip 5mm
\epsfbox{svejk.39} \hskip 5mm
\epsfbox{svejk.41} \vskip 10mm
\epsfbox{svejk.42} \hskip 5mm
\epsfbox{svejk.43}
\end{center}
\caption{The third set of decorated constraints used in Subsection~\ref{sub-subsegments}
         where $(X,Y)\in\{(F,G), (G,F)\}$.}
	 \label{fig-subsegments-3}
\end{figure}

We now analyze the constraints from Figure~\ref{fig-subsegments-3}.
As in the previous cases, the three constraints on the first line force that
each interval $(1-2^{-k+1},1-2^{-k})$, $k\in\NN$, contains disjoint open intervals $I_1,\ldots,I_{t(k)^2}$ and $J_1,\ldots,J_{t(k)^2}$,
each of length $2^{-k}t(k)^{-2}$, such that
$W_0(x,y)=1$ for almost every $(x,y)\in g_F^{-1}(I_i)\times g_F^{-1}(J_i)$ and
$W_0(x,y)=0$ for almost every $(x,y)\in g_F^{-1}(I_i)\times (F_0\setminus g_F^{-1}(J_i))$.

Fix a choice of the roots in the first constraint on the second line in Figure~\ref{fig-subsegments-3};
in almost every choice of the roots, all the three roots belong to the same segment.
Suppose that they belong to the $k$-th segment.
The right side is equal to $\left(2^{-k}t(k)^{-1}\right)^4$ for almost all choices of the roots (since the structure of $B_0^2$ has already been forced) and
the left side is equal to $2^{-2k}\cdot\left(2^{-k}t(k)^{-2}\right)^2$ multiplied by the number of choices of $I_i$ and $J_i$ such that
$I_i$ is contained in the subsegment of the left root and $J_i$ in the subsegment of the right root.
Since the left side and the right side must be equal for almost all choices of the roots,
we conclude that for any pair of subsegments $S$ and $S'$ of the $k$-th segment
there exists a unique index $i$ such that $I_i$ is contained in $S$ and $J_i$ in $S'$.
The last constraint in Figure~\ref{fig-subsegments-3} enforces that
for any fixed subsegment $S$ of the $k$-th segment and any two subsegments $S'$ and $S''$ such that $S'$ precedes $S''$,
the pairs $I_i\times J_i\subseteq S\times S'$ and $I_{i'}\times J_{i'}\subseteq S\times S''$ satisfy that
the interval $I_i$ precedes the interval $I_{i'}$.
This implies that $W_0(x,y)$ and $W_S(g(x),g(y))$ are equal for almost every pair $(x,y)\in F_0\times G_0$.

\subsection{Binary expansions}
\label{sub-binary}

In this section, we force the structure of the graphon inside $E_0\times D_0$, $E_0\times C_0$, $G_0\times C_0$ and $F_0\times C_0$.
This will be achieved using the constraints depicted in Figure~\ref{fig-binary}.

\begin{figure}[htbp]
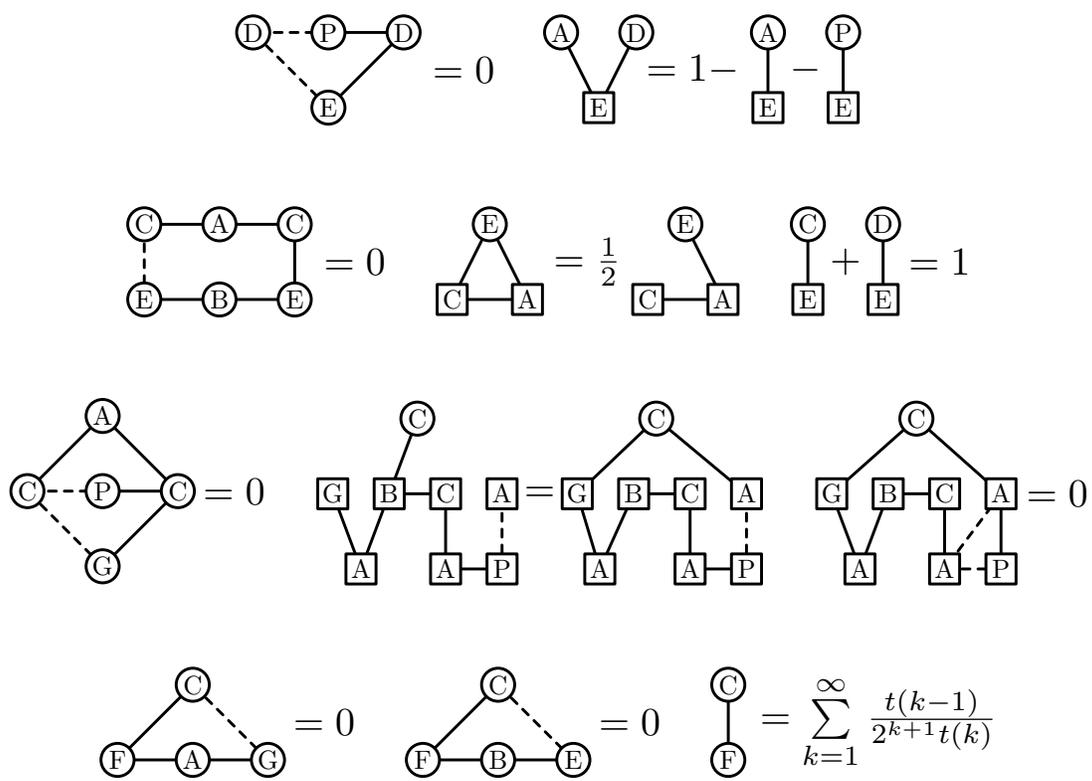

\begin{center}
\epsfbox{svejk.44} \hskip 5mm
\epsfbox{svejk.45} \vskip 10mm
\epsfbox{svejk.46} \hskip 5mm
\epsfbox{svejk.47} \hskip 5mm
\epsfbox{svejk.48} \vskip 10mm
\epsfbox{svejk.49} \hskip 5mm
\epsfbox{svejk.50} \hskip 5mm
\epsfbox{svejk.51} \vskip 10mm
\epsfbox{svejk.52} \hskip 5mm
\epsfbox{svejk.53} \hskip 5mm
\epsfbox{svejk.54}
\end{center}
\caption{The decorated constraints used in Subsection~\ref{sub-binary}.}
\label{fig-binary}
\end{figure}

The first constraint on the first line in the figure causes that for almost every $x\in E_0$, there exists $\xi_x$ such that
$W_0(x,y)=1$ for almost every $y\in D_0$ with $g_D(y)\le\xi_x$ and $W_0(x,y)$ for almost every $y\in D_0$ with $g_D(y)>\xi_x$.
The second constraint on the line causes that for almost every $x\in E_0$,
it holds that $2^{-[g_E(x)]_1} \xi_x = 1-2^{-[g_E(x)]_1}-g_E(x)$.
It follows that
$$\xi_x=\frac{1-2^{-[g_E(x)]_1}-g_E(x)}{2^{-[g_E(x)]_1}}=1-\llbracket g_E(x)\rrbracket_1$$
for almost every $x\in E_0$.
We conclude that $W_0(x,y)$ and $W_S(g(x),g(y))$ are equal for almost every pair $(x,y)\in E_0\times D_0$.

The first constraint on the second line forces that
$W_0$ is almost everywhere $0$ or almost everywhere $1$
on each product of a subsegment of $E_0$ and a segment of $C_0$.
Fix a segment $S_E$ of $E_0$.
The second constraint forces that
for almost every $y\in C_0$
the measure of $x\in S_E$ such that $W_0(x,y)=1$ is exactly half of the measure of $S_E$.
Since the measure of $y\in C_0$ such that $W_0(x,y)=1$ is $1-\xi_x=[g_E(x)]_1$ for almost every $x\in E_0$ because of the third constraint on the second line,
the choice of the segments and subsegments where $W_0$ is one almost everywhere is unique.
So, we get that $W_0(x,y)$ and $W_S(g(x),g(y))$ are equal for almost every pair $(x,y)\in E_0\times C_0$.

The first constraint on the third line implies that for almost every $x\in G_0$ 
there exist $\xi_{x,k}$, $k\in\NN$, such that
$W_0(x,y)=1$ for almost every $y$ with $[g_C(y)]_1=k$ and $\llbracket g_C(y)\rrbracket_1\le\xi_{x,k}$ and
$W_0(x,y)=0$ for almost every other $y\in C_0$.
Consider a possible choice of the roots in the second constraint on the third line.
For almost every such choice of the roots,
if the leftmost root from $A_0$ lies in the $k$-th segment,
then the middle root from $A_0$ is in the $t(k-1)$-th segment (because of the already
enforced structure of $W_0$ on $B_0\times C_0$ in particular) and
the rightmost root from $A_0$ is in the $\ell$-th segment where $\ell\le t(k-1)$.
The left side of the constraint is equal to $2^{-t(k-1)}=t(k)^{-1}$.
The right side of the constraint is equal to $\xi_{x,\ell}2^{-\ell}$.
This implies that $\xi_{x,\ell}=2^\ell/t(k)$ where for almost every $x$ from the $k$-th segment of $G_0$ and $\ell\le t(k-1)$.
Finally, almost every choice of the roots in the last constraint on the third line satisfies that
if the leftmost root from $A_0$ lies in the $k$-th segment,
then the middle root from $A_0$ is in the $t(k-1)$-th segment and
the rightmost root from $A_0$ is in the $\ell$-th segment where $\ell>t(k-1)$.
Hence, $\xi_{x,\ell}=0$ for almost every $x$ from the $k$-th segment of $G_0$ and $\ell>t(k-1)$.
We conclude that $W_0(x,y)$ and $W_S(g(x),g(y))$ are equal for almost every pair $(x,y)\in G_0\times C_0$.

The first constraint on the fourth line implies that for almost every $y\in C_0$,
if the measure of $x$ with $W_0(x,y)>0$ from the $k$-th segment of $F_0$ is positive,
then $W_0(x,y)=1$ for almost every $x$ from the $k$-th segment of $G_0$.
Analogously, the second constraint yields that  for almost every $y\in C_0$,
if the measure of $x$ with $W_0(x,y)>0$ from a certain subsegment of $F_0$ is positive,
then $W_0(x,y)=1$ for almost every $x$ from the corresponding subsegment of $G_0$.
Consequently, $W_0(x,y)=0$ for almost every pair $(x,y)\in F_0\times C_0$ such that $W_S(g(x),g(y)=0$.
Since the last constraint implies that
the integral of $W_0$ over $F_0\times C_0$ is the same as the integral of $W_S$ over $F\times C$,
it holds that $W_0(x,y)$ and $W_S(g(x),g(y))$ are equal for almost every pair $(x,y)\in F_0\times C_0$.

\subsection{Linear transformation}
\label{sub-linear}

In this subsection, we focus on the pair $G_0$ and $E_0$ of the parts.
The first constraint in Figure~\ref{fig-linear} yields that
$W_0(x,y)=0$ for almost every $(x,y)\in G_0\times E_0$ such that the segments of $[g_G(x)]_1\not=[g_E(y)]_1$,
i.e.~the segments of $x$ and $y$ are different.
The second constraint implies that for almost every $x\in G_0$
there exists $\xi_x$ such that
$W_0(x,y)=1$ for almost every $y\in E_0$ such that $[g_G(x)]_1=[g_E(y)]_1$ and $\llbracket g_E(y)\rrbracket_1\ge\xi_x$ and
$W_0(x,y)=0$ for almost every $y\in E_0$ such that $[g_G(x)]_1=[g_E(y)]_1$ and $\llbracket g_E(y)\rrbracket_1<\xi_x$.
The third constraint implies that
almost every pair of $x$ and $x'$ from the same segment of $G_0$ such that $g_G(x)<g_G(x')$
satisfies that $\xi_x\ge\xi_{x'}$.
In order to show that $W_0(x,y)$ and $W_S(g(x),g(y))$ are equal for almost every pair $(x,y)\in G_0\times E_0$,
it is enough to show that
\begin{equation}
\xi_x=\frac{1}{2}+t([g_G(x)]_1-1)^{1/2}\left(\frac{1}{2}-\llbracket g_G(x)\rrbracket_1\right)\label{eq-linear}
\end{equation}
for almost every $x\in G_0$.

\begin{figure}[htbp]
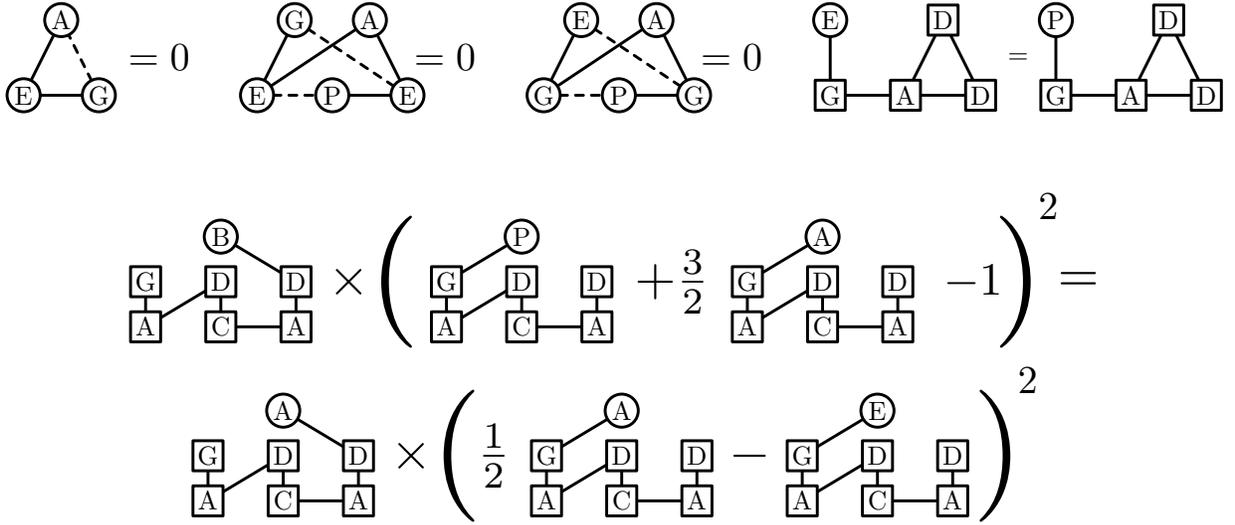

\begin{center}
\epsfbox{svejk.56} \hskip 5mm
\epsfbox{svejk.55} \hskip 5mm
\epsfbox{svejk.60} \hskip 5mm
\epsfbox{svejk.57} \vskip 10mm
\epsfbox{svejk.58} \vskip 2mm
\epsfbox{svejk.59}
\end{center}
\caption{The decorated constraints used in Subsection~\ref{sub-linear}. Note that one of the constraints is on both the second and third lines.}
\label{fig-linear}
\end{figure}

Almost every choice of the roots in the last constraint on the first line in Figure~\ref{fig-linear}
satisfies that the root from $G$ is from the first segment. Hence, this constraint implies that
almost every $x\in G_0$ with $[g_G(x)]_1=1$ satisfies that $\frac{1-\xi_x}{2}=g_G(x)$.
Since $t(0)=1$ and $\llbracket g_G(x)\rrbracket_1=2g_G(x)$ for such $x\in G_0$,
we obtain that (\ref{eq-linear}) holds for almost every $x$ from the first segment of $G_0$.

We now analyze the constraint on the second and third lines in Figure~\ref{fig-linear}.
Note that almost every choice of the roots satisfies that
if the left root from $A$ is in the $k$-th segment,
then the root from $G$ is also in the $k$-th segment and
the right root from $A$ and the right root from $D$ are in the $(k-1)$-th segment.
In particular, $k\ge 2$ for almost every choice of the roots.
Rewriting the densities using the already established structure of the graphon, we obtain that
almost every $x\in G_0$ with $[g_G(x)]_1=k\ge 2$ satisfies that
$$\frac{2^{-(k-1)}}{t(k-1)}\left(g_G(x)+\frac{3}{2}\cdot 2^{-k}-1\right)^2=
  2^{-(k-1)}\left(\frac{1}{2}\cdot 2^{-k}-(1-\xi_x)2^{-k}\right)^2\;\mbox{.}$$
Since it holds that $\llbracket g_G(x)\rrbracket_1=2^k\cdot(g_G(x)-1+2^{-(k-1)})$,
we can rewrite the equation to obtain
$$\frac{2^{-(k-1)}}{t(k-1)}\left(2^{-k}\llbracket g_G(x)\rrbracket_1-\frac{1}{2}\cdot 2^{-k}\right)^2=
  2^{-(k-1)}\left(\frac{1}{2}\cdot 2^{-k}-(1-\xi_x)2^{-k}\right)^2\;\mbox{,}$$
which can be transformed to
$$\left(\llbracket g_G(x)\rrbracket_1-\frac{1}{2}\right)^2=
  t(k-1)\left(\xi_x-\frac{1}{2}\right)^2\;\mbox{.}$$
This implies that $\xi_x$ is equal to
$$\frac{1}{2}+t([g_G(x)]_1-1)^{1/2}\left(\frac{1}{2}-\llbracket g_G(x)\rrbracket_1\right)\mbox{ or }\;
  \frac{1}{2}-t([g_G(x)]_1-1)^{1/2}\left(\frac{1}{2}-\llbracket g_G(x)\rrbracket_1\right)$$
for almost every $x\in G_0$ not contained in the first segment of $G_0$.
Recall that almost every pair of $x$ and $x'$ from the same segment of $G_0$ such that $g_G(x)<g_G(x')$,
which is equivalent to $\llbracket g_G(x)\rrbracket_1<\llbracket g_G(x')\rrbracket_1$, satisfies that $\xi_x\ge\xi_{x'}$.
This implies that the latter of two options for the values $\xi_x$ holds for almost no $x\in G_0$ and
thus (\ref{eq-linear}) also holds for almost every $x\in G_0$ not contained in the first segment of $G_0$.
We conclude that $W_0(x,y)$ and $W_S(g(x),g(y))$ are equal for almost every pair $(x,y)\in G_0\times E_0$.

\subsection{Dot product}
\label{sub-dot}

We now enforce the structure inside the pairs $F_0\times E_0$ and $D_0\times G_0$.
The first constraint in Figure~\ref{fig-dot} yields that
$W_0(x,y)=0$ for almost every $(x,y)\in F_0\times E_0$ with $[g_F(x)]_1\not=[g_E(y)]_1$,
i.e.~the segments of $x$ and $y$ are different.
The second constraint implies that for almost every $x\in F_0$
there exists $\xi_x$ such that
$W_0(x,y)=1$ for almost every $y\in E_0$ with $[g_F(x)]_1=[g_E(y)]_1$ and $\llbracket g_E(y)\rrbracket_1\le\xi_x$ and
$W_0(x,y)=0$ for almost every $y\in E_0$ with $[g_F(x)]_1=[g_E(y)]_1$ and $\llbracket g_E(y)\rrbracket_1>\xi_x$.
In order to show that $W_0(x,y)$ and $W_S(g(x),g(y))$ are equal for almost every pair $(x,y)\in F_0\times E_0$,
it is enough to show that
\begin{equation}
\xi_x=\frac{1}{2} - \frac{\scalar{[ g_F(x)]_2^{\pm1}}{[g_F(x)]_{3}^{\pm1}}}{4t([g_F(x)]_1-1)} \label{eq-dot-1}
\end{equation}
for almost every $x\in F_0$.

\begin{figure}[htbp]
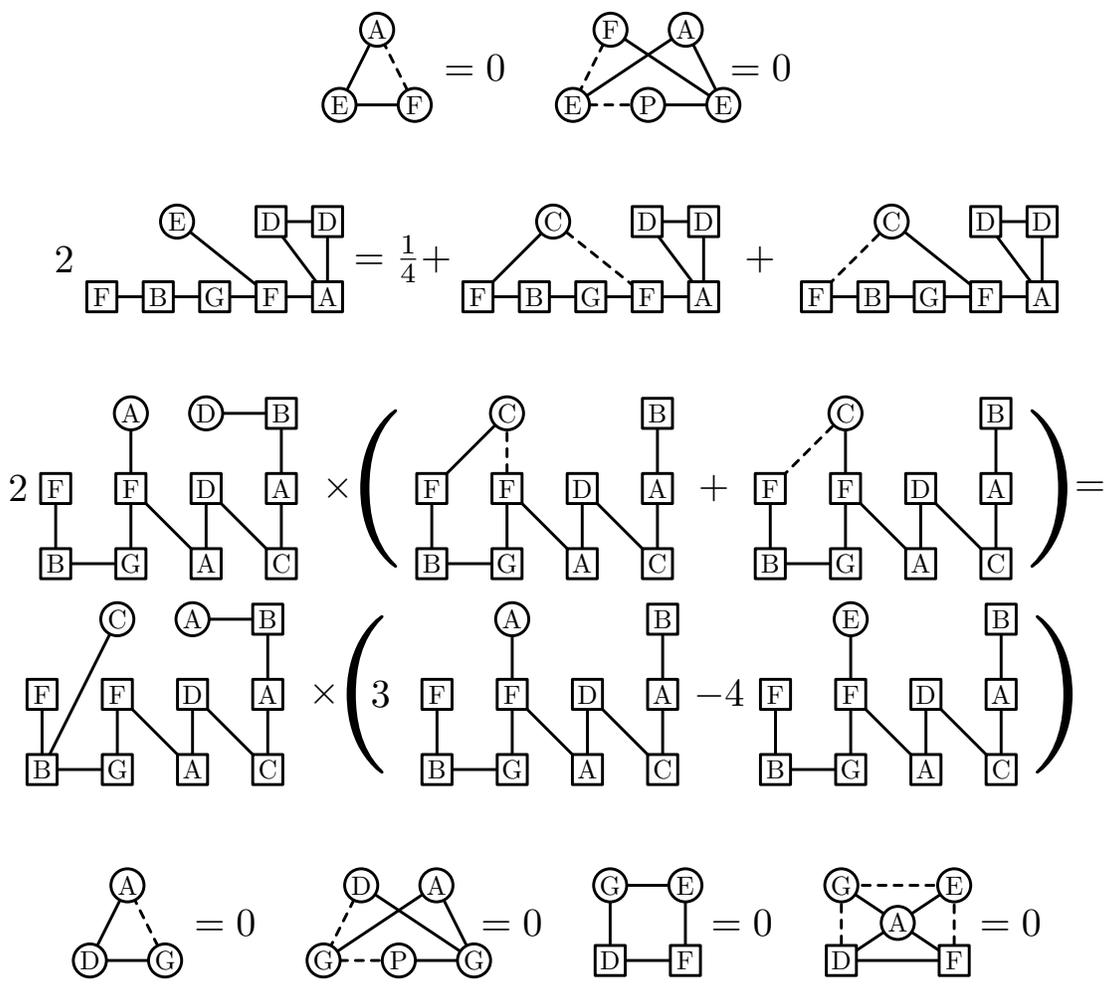

\begin{center}
\epsfbox{svejk.61} \hskip 5mm
\epsfbox{svejk.62} \vskip 10mm
\epsfbox{svejk.63} \vskip 10mm
\epsfbox{svejk.64} \vskip 2mm
\epsfbox{svejk.65} \vskip 10mm
\epsfbox{svejk.66} \hskip 5mm
\epsfbox{svejk.67} \hskip 5mm
\epsfbox{svejk.68} \hskip 5mm
\epsfbox{svejk.69}
\end{center}
\caption{The set of decorated constraints used in Subsection~\ref{sub-dot}.
         Note that one of the constraints is on both the third and fourth lines.}
\label{fig-dot}
\end{figure}

Let $x$ and $x'$ be two elements of $F_0$ from the same segment, say, the $k$-th segment.
By the structure on $F_0\times C_0$,
the measure of $y\in C_0$ such that either $W_0(x,y)=1$ and $W_0(x',y)=0$ or $W_0(x,y)=0$ and $W_0(x',y)=1$
is equal to the number of different pairs of coordinates in $\llbracket g_F(x)\rrbracket_{1}^{\pm1}$ and $\llbracket g_F(x')\rrbracket_{1}^{\pm1}$
multiplied by $t(k)^{-1}$ for almost any pair $x$ and $x'$.
This measure can be rewritten as
\begin{equation}
\frac{t(k-1) - \scalar{\llbracket g_F(x)\rrbracket_{1}^{\pm1}}{\llbracket g_F(x')\rrbracket_{1}^{\pm1}}}{2t(k)}\;\mbox{.}
\label{eq-dot-measure}
\end{equation}
Consider now the constraint on the second line in Figure~\ref{fig-dot}.
For almost every choice of the roots, their segments are the same and so
they all belong to the first segment of their parts (because of the edge between the two roots from $D_0$).
Let $x$ be the right root that belongs to $F_0$ and $x'$ the left one.
Because of the already enforced structure of the graphon,
the subsegment of $x'$ is $[g_F(x)]_{3}$ in almost every choice of the roots.
Using (\ref{eq-dot-measure}) with $t(0)=1$ and $t(1)=2$,
we derive that it holds for almost every $x\in F_0$ that belongs to the first segment that
$$2\cdot\frac{\xi_x}{2}=\frac{1}{4}+\frac{1-\scalar{[ g_F(x)]_2^{\pm1}}{[g_F(x)]_{3}^{\pm1}}}{4}\;\mbox{.}$$
We conclude that the equation (\ref{eq-dot-1}) holds for almost every $x\in F_0$ from the first segment.

We now consider the constraint on the third and fourth lines in Figure~\ref{fig-dot}.
In almost every choice of the roots, the index of the segment of the two roots from $F_0$
is one larger than that of the segment of the root from $A_0$ adjacent to the root from $C_0$.
Also note that for almost every $x\in F_0$ that does not belong to the first segment,
there is a set of positive measure of possible choices of the roots.
Let $x$ be the right root from $F_0$, $x'$ the left one and $k$ the common index of their segment.
The constraint implies for almost every choice of the roots that
$$2\cdot\frac{2^{-k}\cdot 2^{-(k-1)}}{t(k-1)}\cdot\frac{t(k-1)-\scalar{\llbracket g_F(x)\rrbracket_{1}^{\pm1}}{\llbracket g_F(x')\rrbracket_{1}^{\pm1}}}{2t(k)}=
  \frac{2^{-(k-1)}}{2^{t(k-1)}}\left(3\cdot 2^{-k}-4\xi_x\cdot 2^{-k}\right)\;\mbox{.}$$
This expression readily transforms to
$$1-\frac{\scalar{\llbracket g_F(x)\rrbracket_{1}^{\pm1}}{\llbracket g_F(x')\rrbracket_{1}^{\pm1}}}{t(k-1)}=
  3-4\xi_x\;\mbox{.}$$
Since we can choose the roots for almost every $x\in F_0$ that does not belong to the first segment with positive probability,
almost every such $x\in F_0$ satisfies (\ref{eq-dot-1}).
We conclude that $W_0(x,y)$ and $W_S(g(x),g(y))$ are equal for almost every pair $(x,y)\in F_0\times E_0$.

In the analogy to the first two constraints on the first last in Figure~\ref{fig-dot},
the first two constraints on the last line in the figure yield that
for almost every $z\in D_0$, there exists $\lambda_z$ such that
$W_0(z,y)=1$ for almost every $y\in G_0$ with $[g_D(z)]_1=[g_G(y)]_1$ and $\llbracket g_G(y)\rrbracket_1\le\lambda_z$ and
$W_0(z,y)=0$ for almost every $y\in G_0$ with $[g_D(z)]_1\not=[g_G(y)]_1$ or $\llbracket g_G(y)\rrbracket_1>\lambda_z$.
We now consider the last two constraints on the last line in Figure~\ref{fig-dot}.
Fix a choice of roots $x\in F_0$ and $z\in D_0$.
Almost every such choice of roots satisfies $[g_F(x)]_{2,3}=[g_D(z)]_{2,3}$, which implies that $[g_F(x)]_1=[g_D(z)]_1$.
The last but one constraint yields that $W_0(y,y')=0$ for almost every $y\in E_0$, $y'\in G_0$,
$[g_E(y)]_1=[g_F(x)]_1$, $[g_G(y)]_1=[g_D(z)]_1$,
$\llbracket g_E(y)\rrbracket_1\le\xi_x$ and $\llbracket g_G(y')\rrbracket_1\le\lambda_z$.
Similarly, the last constraint yields that $W_0(y,y')=1$ for almost every $y\in E_0$, $y'\in G_0$,
$[g_E(y)]_1=[g_F(x)]_1$, $[g_G(y)]_1=[g_D(z)]_1$,
$\llbracket g_E(y)\rrbracket_1>\xi_x$ and $\llbracket g_G(y')\rrbracket_1>\lambda_z$.
The structure of the graphon $W_0$ on $E_0\times G_0$ implies that
\begin{equation}
\lambda_z=\trunc{t(k-1)^{1/2}\left(\frac{1}{2}-\xi_x\right)+\frac{1}{2}}
\label{eq-dot-2}
\end{equation}
for almost every pair of $x\in F_0$ and $z\in D_0$ such that
$[g_F(x)]_{2,3}=[g_D(z)]_{2,3}$ and $[g_F(x)]_1=[g_D(z)]_1=k$.
The expression (\ref{eq-dot-1}) can be rewritten (for a particular choice of $x$ and $z$) as
\begin{equation}
\xi_x=\frac{1}{2} - \frac{\scalar{[ g_D(z)]_2^{\pm1}}{[g_D(z)]_{3}^{\pm1}}}{4t(k-1)}\;\mbox{.}
\label{eq-dot-3}
\end{equation}
We get by substituting (\ref{eq-dot-3}) in (\ref{eq-dot-2}) that
$$\lambda_z=\trunc{\frac{1}{2}+\frac{\scalar{[ g_D(z)]_2^{\pm1}}{[g_D(z)]_{3}^{\pm1}}}{4t(k-1)^{1/2}}}$$
for almost every $z\in D_0$ that belongs to the $k$-th segment.
We conclude that $W_0(z,y)$ and $W_S(g(z),g(y))$ are equal for almost every pair $(z,y)\in D_0\times G_0$.

\subsection{Main part}
\label{sub-main}

We now focus on the values of $W_0$ on $E_0^2$. The first constraint in Figure~\ref{fig-main} implies that
$W_0(x,y)=0$ for almost every $x,y\in E_0^2$ with $[g_E(x)]_1\not=[g_E(y)]_1$.
Let us consider the second constraint in the figure.
For almost every choice of the roots, all the roots belong to the same segment in their respective parts.
Let $k$ be the index of this segment and let $z$ and $z'$ be the choices of the bottom and the top roots that belong to $B_0$.
Further, let $S$ and $S'$ be the subsegments of $E_0$ corresponding to the subsegments of $z$ and $z'$ in $B_0$, respectively.
For almost every choice of the roots, the root from $E_0$ belongs to $S$.
The second constraint yields that it holds for almost all $x\in S$ that
\begin{equation}
2^{-k}\int\limits_{S'}W(x,y)\dd y=|S'|\cdot\frac{\trunc{\frac{1}{2} + \frac{\scalar{\llbracket g_B(z)\rrbracket_{1}^{\pm1}}{\llbracket g_B(z')\rrbracket_{1}^{\pm1}}}{4t(k-1)^{1/2}}}}{2^k}\;\mbox{.}
\label{eq-main-1}
\end{equation}
Along the same lines, the third constraint yields that it holds for almost all $x,x'\in S$ that
\begin{equation}
2^{-2k}\int\limits_{S'}W(x,y)W(x',y)\dd y=|S'|\frac{\left(\trunc{\frac{1}{2} + \frac{\scalar{\llbracket g_B(z)\rrbracket_{1}^{\pm1}}{\llbracket g_B(z')\rrbracket_{1}^{\pm1}}}{4t(k-1)^{1/2}}}\right)^2}{2^{2k}}\;\mbox{.}
\label{eq-main-2}
\end{equation}
By Lemma~\ref{lm-square}, the identity (\ref{eq-main-2}) implies that
\begin{equation}
\frac{1}{|S'|}\int\limits_{S'}W(x,y)^2\dd y=\left(\trunc{\frac{1}{2} + \frac{\scalar{\llbracket g_B(z)\rrbracket_{1}^{\pm1}}{\llbracket g_B(z')\rrbracket_{1}^{\pm1}}}{4t(k-1)^{1/2}}}\right)^2\;\mbox{.}
\label{eq-main-3}
\end{equation}
for almost every $x\in S$.

\begin{figure}[htbp]
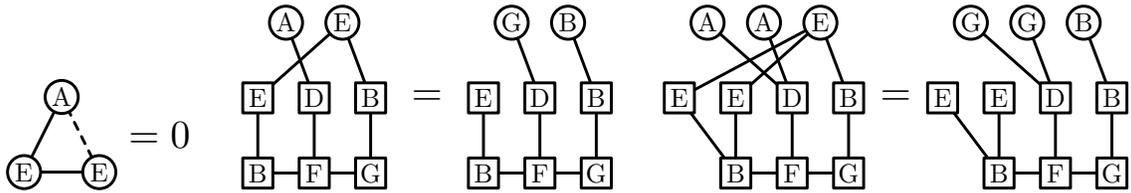

\begin{center}
\epsfbox{svejk.70} \hskip 5mm
\epsfbox{svejk.71} \hskip 5mm
\epsfbox{svejk.72}
\end{center}
\caption{The set of decorated constraints used in Subsection~\ref{sub-main}.}
\label{fig-main}
\end{figure}

We derive from (\ref{eq-main-1}) and (\ref{eq-main-3}) using Cauchy-Schwarz inequality that
$$W(x,y)=\trunc{\frac{1}{2} + \frac{\scalar{\llbracket g_B(z)\rrbracket_{1}^{\pm1}}{\llbracket g_B(z')\rrbracket_{1}^{\pm1}}}{4t(k-1)^{1/2}}}$$
for almost every $x\in S$ and every $y\in S'$.
Since $\llbracket g_B(z)\rrbracket_{1}^{\pm1}=\llbracket g_E(x)\rrbracket_{1}^{\pm1}$ and $\llbracket g_B(z')\rrbracket_{1}^{\pm1}=\llbracket g_E(y)\rrbracket_{1}^{\pm1}$
for almost every such pair $x$ and $y$,
we conclude that $W_0(x,y)$ and $W_S(g(x),g(y))$ are equal for almost every pair $(x,y)\in E_0^2$.

\subsection{Degree balancing}
\label{sub-degrees}

Fix $X\in\{A,\ldots,G,P\}$.
The first constraint in Figure~\ref{fig-degrees} implies that
$$\int\limits_{X_0} W_0(x,y)W_0(x',y)\dd y=K_X$$
for almost every $x,x'\in Q_0$ for some $K_X$.
By Lemma~\ref{lm-square}, this also implies that
$$\int\limits_{X_0} W_0(x,y)^2\dd y=K_X$$
for almost every $x\in Q_0$.
Hence, there exists a function $h_X:Q_0\to [0,1]$ such that $W_0(x,y)=h_X(y)$ for almost every $x\in Q_0$ and $y\in X_0$.
Since the last constraint on the first line yields that
$$\int\limits_{\overline{R_0}}W_0(x,y)\dd y=\frac{4}{13}$$
for almost every $x\in X_0$ and it holds that
$$\int\limits_{\overline{Q_0\cup R_0}}W_0(x,y)\dd y=\int\limits_{\overline{Q\cup R}}W_S(g_X(x),y)\dd y$$
for almost every $x\in X_0$ because of the already enforced structure,
we conclude that $W_0(x,y)$ and $W_S(g(x),g(y))$ are equal for almost every pair $(x,y)\in X_0\times Q_0$.

\begin{figure}[htbp]
\begin{center}
\epsfbox{svejk.73} \hskip 5mm
\epsfbox{svejk.74} \vskip 10mm
\epsfbox{svejk.75} \hskip 5mm
\epsfbox{svejk.76} \hskip 5mm
\epsfbox{svejk.77} \vskip 10mm
\epsfbox{svejk.78} \hskip 5mm
\epsfbox{svejk.79}
\end{center}
\caption{The set of decorated constraints used in Subsection~\ref{sub-degrees} where $X\in\{A,B,C,D,E,F,G,P\}$.}
\label{fig-degrees}
\end{figure}

The three constraints on the second line in Figure~\ref{fig-degrees} clearly implies that
$W_0(x,y)$ and $W_S(g(x),g(y))$ are equal for almost every pair $(x,y)\in (Q_0\cup R_0)^2$.

Again fix $X\in\{A,\ldots,G,P\}$.
The two constraints on the last line in Figure~\ref{fig-degrees} give using Lemma~\ref{lm-square} that
$$ \int_{R_0}W_0(x,y)\dd y=K_X|R_0|\; \mbox{ and } \int_{R_0}W_0(x,y)^2\dd y=K_X^2|R_0| $$
for almost every $x\in X_0$ for some $K_X$. 
However, this is only possible if $W_0(x,y)=K_X$ for almost every $x\in X_0$ and almost every $y\in R_0$.
The right side values of the two constraints yield that the values of $K_X$ matches the corresponding values in $W_S$.
So, we can conclude that $W_0(x,y)$ and $W_S(g(x),g(y))$ are equal for almost every pair $(x,y)\in X_0\times R_0$.
Since we have shown that the constraints in Figures~\ref{fig-coordinate}--\ref{fig-degrees} imply that
$W_0(x,y)$ and $W_S(g(x),g(y))$ are equal for almost every pair $x$ and $y$,
we have established that the \svejk graphon is finitely forcible. This completes the proof of Theorem~\ref{thm-main}. 

\section*{Concluding remark}

Proposition~\ref{prop-lower-graphon} gives that it is not possible
to remove $2^{5\log^*\varepsilon_i^{-2}}$ completely from the denominator in the exponent in Theorem~\ref{thm-main}.
However, our construction can be modified to replace $t(n)$ with a faster growing function of $n$, e.g.~with $t(t(n))$,
which would replace the function $2^{5\log^*\varepsilon_i^{-2}}$ with a slower growing function of $\varepsilon^{-1}$.
In fact, a recent result from~\cite{bib-ff} implies that $2^{5\log^*\varepsilon_i^{-2}}$
can be replaced with any Turing machine computable function of $\varepsilon^{-1}$ growing to infinity.

\section*{Acknowledgement}

The authors are grateful to Jacob Fox for sparking their interest in finitely forcible graphons
with weak regular partitions requiring a large number of parts and
for pointing out the tightness of their construction as formulated in Proposition~\ref{prop-lower-graphon}.
They would also like to thank Andrzej Grzesik, Sune Jakobsen and Fiona Skerman for discussions on the limit of the graphons $W^{\CF}_m$, and
Taisa Lopes Martins for her comments on one of the drafts of the paper.
Finally, they would like to thank the anonymous reviewer for her/his comments that
have helped to clarify and improve various aspects of the presentation of the results.

\newcommand{\advances}{Adv. Math. }
\newcommand{\annals}{Ann. of Math. }
\newcommand{\cpc}{Combin. Probab. Comput. }
\newcommand{\dcg}{Discrete Comput. Geom. }
\newcommand{\discrete}{Discrete Math. }
\newcommand{\eur}{European J.~Combin. }
\newcommand{\gfa}{Geom. Funct. Anal. }
\newcommand{\jcta}{J.~Combin. Theory Ser.~A }
\newcommand{\jctb}{J.~Combin. Theory Ser.~B }
\newcommand{\rsa}{Random Structures Algorithms }
\newcommand{\sidma}{SIAM J.~Discrete Math. }
\newcommand{\tcs}{Theoret. Comput. Sci. }


\begin{thebibliography}{99}
\bibitem{bib-flag1}
R.~Baber:
{\em Tur\'an densities of hypercubes\/},
available as arXiv:1201.3587.

\bibitem{bib-flagrecent}
R.~Baber and J.~Talbot:
{\em A solution to the $2/3$ conjecture\/},
\sidma {\bf 28} (2014), 756-–766.

\bibitem{bib-flag2}
R.~Baber and J.~Talbot:
{\em Hypergraphs do jump\/},
\cpc {\bf 20} (2011), 161--171.

\bibitem{bib-flag3}
J.~Balogh, P.~Hu, B.~Lidick\'y and H.~Liu:
{\em Upper bounds on the size of 4- and 6-cycle-free subgraphs of the hypercube\/},
\eur {\bf 35} (2014), 75--85.

\bibitem{bib-bollobas11+}
B.~Bollob\'as and O.~Riordan:
{\em Sparse graphs: Metrics and random models\/},
\rsa {\bf 39} (2011), 1--38.

\bibitem{bib-borgs10+}
C.~Borgs, J.T.~Chayes and L.~Lov{\'a}sz:
{\em Moments of two-variable functions and the uniqueness of graph limits\/},
\gfa {\bf 19} (2010), 1597--1619.

\bibitem{bib-borgs08+}
C.~Borgs, J.T.~Chayes, L.~Lov{\'a}sz, V.T.~S{\'o}s and K.~Vesztergombi:
{\em Convergent sequences of dense graphs I: Subgraph frequencies, metric properties and testing\/},
\advances {\bf 219} (2008), 1801--1851.

\bibitem{bib-borgs+}
C.~Borgs, J.T.~Chayes, L.~Lov\'asz, V.T.~S\'os and K.~Vesztergombi:
{\em Convergent sequences of dense graphs II. Multiway cuts and statistical physics\/},
\annals {\bf 176} (2012), 151--219.

\bibitem{bib-borgs06+}
C.~Borgs, J.~Chayes, L.~Lov{\'a}sz, V.T.~S{\'o}s, B.~Szegedy and K.~Vesztergombi:
{\em Graph limits and parameter testing\/},
in: Proceedings of the 38rd Annual ACM Symposium on the Theory of Computing (STOC), ACM, New York, 2006, 261--270.

\bibitem{bib-chung89+}
F.R.K.~Chung, R.L.~Graham and R.M.~Wilson:
{\em Quasi-random graphs\/},
Combinatorica {\bf 9} (1989), 345--362.

\bibitem{bib-conlon12+}
D.~Conlon and J.~Fox:
{\em Bounds for graph regularity and removal lemmas\/},
\gfa {\bf 22} (2012), 1191--1256.

\bibitem{bib-ff}
J.~W.~Cooper, D.~Kr\'al' and T. Martins:
{\em Finite forcibility and computability of graphons\/},
in preparation.

\bibitem{bib-diaconis09+}
P.~Diaconis, S.~Holmes and S.~Janson:
{\em Threshold graph limits and random threshold graphs\/},
Internet Math. {\bf 5} (2009), 267--318.

\bibitem{bib-elek07}
G.~Elek:
{\em On limits of finite graphs\/},
Combinatorica {\bf 27} (2007), 503--507.

\bibitem{bib-frieze99+}
A.~Frieze and R.~Kannan:
{\em Quick approximation to matrices and applications\/},
Combinatorica {\bf 19}, 175--220.

\bibitem{bib-glebov-test}
R. Glebov, C. Hoppen, T.~Klimo\v sov\'a, Y. Kohayakawa, D. Kr\'al' and H. Liu:
{\em Large permutations and parameter testing\/},
available as arXiv:1412.5622.

\bibitem{bib-inf}
R.~Glebov, T.~Klimo\v sov\'a and D.~Kr\'al':
{\em Infinite dimensional finitely forcible graphon\/},
available as arXiv:1404.2743.

\bibitem{bib-comp}
R.~Glebov, D.~Kr\'al' and J.~Volec:
{\em Compactness and finite forcibility of graphons\/},
available as arXiv:1309.6695.

\bibitem{bib-flag4}
A.~Grzesik:
{\em On the maximum number of five-cycles in a triangle-free graph\/},
\jctb {\bf 102} (2012), 1061--1066.

\bibitem{bib-flag5}
H.~Hatami, J.~Hladk\'y, D.~Kr\'al', S.~Norine and A.~Razborov:
{\em Non-three-colorable common graphs exist\/},
\cpc {\bf 21} (2012), 734--742.

\bibitem{bib-flag6}
H.~Hatami, J.~Hladk\'y, D.~Kr\'al', S.~Norine and A.~Razborov:
{\em On the number of pentagons in triangle-free graphs\/},
\jcta {\bf 120} (2013), 722--732.

\bibitem{bib-hoppen-lim1}
C.~Hoppen, Y.~Kohayakawa, C.G.~Moreira, B.~R\'ath and R.M.~Sampaio:
{\em Limits of permutation sequences\/},
\jctb {\bf 103} (2013), 93--113.

\bibitem{bib-hoppen-lim2}
C.~Hoppen, Y.~Kohayakawa, C.G.~Moreira and R.M.~Sampaio:
{\em Limits of permutation sequences through permutation regularity\/},
available as arXiv:1106.1663.

\bibitem{bib-hoppen-test}
C.~Hoppen, Y.~Kohayakawa, C.G.~Moreira and R.M.~Sampaio:
{\em Testing permutation properties through subpermutations\/},
\tcs {\bf 412} (2011), 3555--3567.

\bibitem{bib-janson11}
S.~Janson:
{\em Poset limits and exchangeable random posets\/},
Combinatorica {\bf 31} (2011), 529--563.

\bibitem{bib-flag7}
D.~Kr\'al', C.-H.~Liu, J.-S.~Sereni, P.~Whalen and Z.~Yilma:
{\em A new bound for the 2/3 conjecture\/},
\cpc {\bf 22} (2013), 384--393.

\bibitem{bib-flag8}
D.~Kr\'al', L.~Mach and J.-S.~Sereni:
{\em A new lower bound based on Gromov’s method of selecting heavily covered points\/},
\dcg {\bf 48} (2012), 487--498.

\bibitem{bib-kral12+}
D.~Kr{\'a}l' and O.~Pikhurko:
{\em Quasirandom permutations are characterized by 4-point densities\/},
\gfa {\bf 23} (2013), 570--579.

\bibitem{bib-lovasz-book}
L.~Lov\'asz:
Large networks and graph limits,
AMS, Providence, RI, 2012.

\bibitem{bib-lovasz08+}
L.~Lov\'asz and V.T.~S\'os:
{\em Generalized quasirandom graphs\/},
\jctb {\bf 98} (2008), 146--163.

\bibitem{bib-lovasz11+}
L.~Lov{\'a}sz and B.~Szegedy:
{\em Finitely forcible graphons\/},
\jctb {\bf 101} (2011), 269--301.

\bibitem{bib-lovasz06+}
L.~Lov{\'a}sz and B.~Szegedy:
{\em Limits of dense graph sequences\/},
\jctb {\bf 96} (2006), 933--957.

\bibitem{bib-lovasz10+}
L.~Lov{\'a}sz and B.~Szegedy:
{\em Testing properties of graphs and functions\/},
Israel J. Math. {\bf 178} (2010), 113--156.

\bibitem{bib-flag9}
O.~Pikhurko and A.~Razborov:
{\em Asymptotic structure of graphs with the minimum number of triangles\/},
available as arXiv:1204.2846.

\bibitem{bib-flag10}
O.~Pikhurko and E.R.~Vaughan:
{\em Minimum number of k-cliques in graphs with bounded independence number\/},
\cpc {\bf 22} (2013), 910--934.

\bibitem{bib-razborov07}
A.~Razborov:
{\em Flag algebras\/},
J. Symbolic Logic {\bf 72} (2007), 1239--1282.

\bibitem{bib-flag11}
A.~Razborov:
{\em On 3-hypergraphs with forbidden 4-vertex configurations\/},
\sidma {\bf 24} (2010), 946--963.

\bibitem{bib-flag12}
A.~Razborov:
{\em On the minimal density of triangles in graphs\/},
\cpc {\bf 17} (2008), 603--618.

\bibitem{bib-rodl}
V. R\"odl:
{\em On universality of graphs with uniformly distributed edges\/},
\discrete {\bf 59} (1986), 125--134.

\bibitem{bib-thomason}
A. Thomason:
{\em Pseudo-Random Graphs\/},
in: A. Barlotti, M. Biliotti, A. Cossu, G. Korchmaros and G. Tallini, (eds.), North-Holland Mathematics Studies, North-Holland, 1987, {\bf  144}, 307--331.

\bibitem{bib-thomason2}
A. Thomason:
{\em Random graphs, strongly regular graphs and pseudorandom graphs\/},
in: C. Whitehead (ed.), Surveys in Combinatorics 1987, London Mathematical Society Lecture Note Series, 123, Cambridge Univ. Press, Cambridge (1987), 173–-195.

\end{thebibliography}
\end{document}